\documentclass[a4paper,10pt]{amsart}

\usepackage{verbatim, enumerate, tensor}

\DeclareMathOperator{\id}{id}
\DeclareMathOperator{\Ker}{Ker \,}
\DeclareMathOperator{\Span}{Span}
\DeclareMathOperator{\Tr}{Tr \,}
\DeclareMathOperator{\SO}{SO}
\DeclareMathOperator{\SL}{SL}
\DeclareMathOperator{\GL}{GL}
\DeclareMathOperator{\diag}{diag}

\DeclareMathOperator{\Aut}{Aut}
\DeclareMathOperator{\ad}{ad}
\DeclareMathOperator{\ric}{Ric}
\DeclareMathOperator{\Der}{Der}
\DeclareMathOperator{\rk}{rk}
\DeclareMathOperator{\End}{End}
\DeclareMathOperator{\Mat}{Mat}

\renewcommand \a{\alpha}

\newcommand \K{\delta}

\newcommand \la{\lambda}
\newcommand \Rn{\mathbb R^n}
\newcommand \bc{\mathbb{C}}
\newcommand \br{\mathbb{R}}

\newcommand \<{\langle}
\renewcommand \>{\rangle}
\newcommand \ip{\<\cdot,\cdot\>}

\newcommand \g{\mathfrak{g}}

\renewcommand \sl{\mathfrak{sl}}

\newcommand \n{\mathfrak{n}}
\newcommand \m{\mathfrak{m}}
\newcommand \z{\mathfrak{z}}
\newcommand \reg{\mathfrak{r}}
\renewcommand \b{\mathfrak{b}}

\newcommand \f{\mathfrak{f}}
\renewcommand \t{\mathfrak{t}}

\newcommand\sk[1]{\tensor[^s]{{#1}}{}}

\theoremstyle{plane}
\newtheorem{theorem}{Theorem}
\newtheorem*{theorem*}{Theorem}
\newtheorem*{corollary*}{Corollary}
\newtheorem{lemma}{Lemma}
\newtheorem{proposition}{Proposition}
\newtheorem*{proposition*}{Proposition}

\newtheorem*{namedtheorem}{\theoremname}
\newcommand{\theoremname}{te}

\theoremstyle{definition}
\newtheorem{definition}{Definition}

\theoremstyle{remark}
\newtheorem{remark}{Remark}

\begin{document}

\title{Einstein solvmanifolds attached to two-step nilradicals}

\author{Y.Nikolayevsky}
\address{Department of Mathematics, La Trobe University, Victoria, 3086, Australia}
\email{y.nikolayevsky@latrobe.edu.au}

\date{\today}



\maketitle

\begin{abstract}
A Riemannian Einstein solvmanifold (possibly, any noncompact homogeneous Einstein space) is almost completely
determined by the nilradical of its Lie algebra. A nilpotent Lie algebra, which can serve as the nilradical of an
Einstein metric solvable Lie algebra, is called an Einstein nilradical. Despite
a substantial progress towards the understanding of Einstein nilradicals, there is still a lack
of classification results even for some well-studied classes of nilpotent Lie algebras,
such as the two-step ones. In this paper, we give a classification of two-step nilpotent Einstein nilradicals
in one of the rare cases when the complete set of affine invariants is known: for the two-step nilpotent Lie algebras
with the two-dimensional center. Informally speaking, we prove that such a Lie algebra is an Einstein nilradical,
if it is defined by a matrix pencil having no nilpotent blocks in the canonical form and no elementary divisors of a
very high multiplicity. We also discuss the connection between the property of a two-step nilpotent Lie algebra
and its dual to be an Einstein nilradical.
\end{abstract}

\section{Introduction}
\label{s:intro}

The theory of Riemannian homogeneous spaces with an Einstein metric splits into three very different cases
depending on the sign of the Einstein constant, the scalar curvature. Among them, the picture is complete only in the
Ricci-flat case: by the result of \cite{AK}, every Ricci-flat homogeneous space is flat.

The major open conjecture in the case of negative scalar curvature is the \emph{Alekseevski Conjecture} \cite{Al1}
asserting that a noncompact Einstein homogeneous space admits a simply transitive solvable isometry group. This is
equivalent to saying that any such space is a \emph{solvmanifold}, a solvable Lie group with a left-invariant Riemannian
metric satisfying the Einstein condition.

By a deep result of J.Lauret \cite{La5}, any Einstein solvmanifold is \emph{standard}. This means that the metric solvable
Lie algebra $\g$ of such a solvmanifold has the following property: the orthogonal complement to the derived algebra
of $\g$ is abelian. The systematic study of standard Einstein solvmanifolds (and the term ``standard") originated from
the paper of J.Heber \cite{Heb}.

On the Lie algebra level, all the metric Einstein solvable Lie algebras can be obtained as the result of the following
construction \cite{Heb, La1, La5, LW}. One starts with the three pieces of data: a nilpotent Lie algebra $\n$, a
semisimple derivation $\Phi$ of $\n$, and an inner product $\ip_\n$ on $\n$, with respect to which $\Phi$ is
symmetric. An extension of $\n$ by $\Phi$ is a solvable Lie algebra $\g=\br H \oplus \n$ (as a linear space) with
$(\ad_H)_{|\n}:=\Phi$. The inner product on $\g$ is defined by $\<H, \n\> = 0$, $\|H\|^2 = \Tr \Phi$ (and coincides with
the existing one on $\n$). The resulting metric solvable Lie algebra $(\g,\ip)$ is Einstein provided
$\n$ is ``nice" and the derivation $\Phi$ and the inner product $\ip_\n$ are
chosen ``in the correct way". Metric Einstein solvable Lie algebras of higher rank
(with the codimension of the nilradical greater than one) having the same nilradical $\n$ can be obtained from $\g$ via
a known procedure, by further adjoining to $\n$ semisimple derivation commuting with $\Phi$.

It turns out that the structure of an Einstein metric solvable Lie algebra is completely encoded in its nilradical in
the following sense: given a nilpotent Lie algebra $\n$, there is no more than one (possibly none) choice of $\Phi$ and
of $\ip_\n$, up to conjugation by $\Aut(\n)$ and scaling, which may result in an Einstein metric solvable Lie
algebra $(\g,\ip)$.

\begin{definition} \label{d:en}
A nilpotent Lie algebra is called an \emph{Einstein nilradical}, if it is the nilradical of an Einstein metric solvable
Lie algebra. A derivation $\Phi$ of an Einstein nilradical $\n$ and an inner product $\ip_\n$, for which
the metric solvable Lie algebra $(\g,\ip)$ is Einstein are called an \emph{Einstein derivation} and a
\emph{nilsoliton} inner product respectively.
\end{definition}

To a certain extent, the problem of classifying Einstein solvmanifolds (possibly, all noncompact Einstein
homogeneous spaces) is the problem of classifying Einstein nilradicals. The condition expressing the fact that a
given nilpotent Lie algebra is an Einstein nilradical is a system of algebraic equations (see Section~\ref{s:facts}),
which is quite difficult to analyze. Although a direct attack on that system
(solving it explicitly or proving the existence of the solution) can be successful for some classes of nilpotent Lie
algebras, it might not be very efficient in general.

An alternative approach to classifying Einstein nilradicals (or to deciding, whether
a given nilpotent Lie algebra is an Einstein nilradical) is the variational one. It originates from the fact
that all the Einstein Riemannian metrics on a compact manifold (not necessarily homogeneous) are the critical points of
the Hilbert functional, the total scalar curvature, on the space of metrics of volume $1$ \cite[Chapter~4]{Bes}.
For left-invariant metrics on a homogeneous manifold, the scalar curvature is
constant, so the Einstein metrics in the compact case can be found as the critical points of the scalar curvature
functional $sc$ on the space $\SL(n)/\SO(n)$ of the inner products of volume $1$ on the tangent space at the basepoint.
The same functional $sc$ also works fine for noncompact homogeneous spaces of unimodular groups
(see \cite{Jen} for the groups, \cite{Nik} for the general case). However, the solvable
case is different. As shown in \cite[Lemma~3.5]{Heb}, the functional $sc$ (and even its modifications) has no
critical points on the space of the inner products of the fixed volume on a solvable Lie algebra. One may try to restrict
one's attention to the nilradical and to hope that the critical points of $sc$ detect the nilsoliton inner
products on nilpotent Lie algebras. Unfortunately, this also fails: as it follows from \cite[Theorem 8.2]{La3},
$sc$ has no critical points on the space of the inner products of the fixed volume on a nonabelian Lie algebra. However,
the nilsoliton inner products on a nilpotent Lie algebra \emph{are} the critical points of the normalized squared norm
of the Ricci tensor \cite[Theorem~2.2]{La4}. This functional
is, in fact, the squared norm of the moment map of the $\SL(n)$-action on the space of metrics (more precisely, on the
space of Lie brackets), which enables one to use the powerful machinery of the Geometric Invariant Theory to study
Einstein nilradicals \cite{LW}. Another variational approach based on  \cite[Theorem 6.15]{Heb} is developed
in \cite{Ni4}: one still
uses the functional $sc$, but restricted to a smaller group $G_\phi \subset \SL(n)$. The group $G_\phi$ is constructed
from the so-called \emph{pre-Einstein derivation}, which is a certain distinguished semisimple derivation of a given
nilpotent Lie algebra (see Section~\ref{ss:es} for details).

\bigskip

In addition to general classification results obtained by the variational method, one might want to
have a reasonable supply of examples of nilpotent Lie algebras, which are, or which are not, Einstein nilradicals.
It seems to be a commonly accepted opinion that the general classification of nilpotent
Lie algebras, in any reasonable meaning, is out of question. However, it would be interesting to have a description
(or even a classification) of Einstein nilradicals in those classes of nilpotent Lie algebras, for which
such a classification is possible.

It is known, for instance, that the following nilpotent Lie algebras $\n$ are Einstein nilradicals:
$\n$ is abelian \cite{Al2}, $\n$ has a codimension one abelian ideal \cite{La2},
$\dim \n \le 6$ \cite{Wil, La2}, $\n$ is the algebra of strictly upper-triangular matrices \cite{Pay}
(see \cite[Section~6]{La4} for many other examples).
Free Einstein nilradical are classified in \cite{Ni2}: apart from the abelian and the two-step ones,
there are only six others. A characterization of Einstein nilradical with a simple Einstein derivation and
the classification of filiform Einstein nilradicals
(modulo known classification of filiform $\mathbb{N}$-graded Lie algebras) is given in \cite{Ni3}. Einstein
nilradicals of parabolic subalgebras of semisimple Lie algebras are studied in \cite{Tam}.

\bigskip

In this paper, we focus on two-step nilpotent Lie algebras. As it was independently proved
by P.~Eberlein \cite[Proposition~7.9]{Eb3} and the author \cite[Theorem~4]{Ni4}, \emph{a typical two-step nilpotent
Lie algebra is an Einstein nilradical of eigenvalue type $(1,2; q,p)$}. More precisely,
a two-step nilpotent Lie algebra $\n$ is said to be of type $(p, q)$, if $\dim \n = p+q$ and
$\dim [\n, \n] = p$ (clearly, $1 \le p \le D:=\frac12 q(q-1)$). Any such algebra is determined by a point in the
linear space $\mathcal{V}(p,q)=(\bigwedge^2 \br^q)^p$, with two points giving isomorphic algebras if and only if they
lie on the same orbit of the action of $\GL(q) \times \GL(p)$ on $\mathcal{V}(p,q)$. The space of isomorphism classes
of algebras of type $(p,q)$ is a compact non-Hausdorff space $\mathcal{X}(p,q)$,
the quotient of an open and dense subset of
$\mathcal{V}(p,q)$ by the action of $\GL(q) \times \GL(p)$.

\begin{proposition}[\cite{Eb3}, \cite{Ni4}]\label{p:twostepopen} Suppose $q \ge 6$ and $2 < p < \frac12 q(q-1)-2$,
or $(p,q)=(5,5)$. Then
\begin{enumerate}[\rm (i)]
    \item
    there is a continuum isomorphism classes of two-step nilpotent Lie algebras of type $(p,q)$; each of them
    has an empty interior in the space $\mathcal{V}(p,q)$ \emph{\cite{Eb1}};
    \item
    the space $\mathcal{V}(p,q)$ contains an open and dense subset corresponding to two-step nilpotent Einstein
    nilradicals of eigenvalue type $(1,2; q,p)$.
\end{enumerate}
\end{proposition}

Note that the sets of exceptional pairs $(p,q)$ in \cite{Eb3} is smaller then that in \cite{Ni4}. The difference is
explained in Remark~\ref{rem:typical} in Section~\ref{ss:2step}.

Proposition~\ref{p:twostepopen} misses two-step nilpotent Lie algebras $\n$ with the following $(p,q)$:
\begin{itemize}
    \item $p=1$. Any such $\n$ is the direct sum of a Heisenberg algebra and an abelian ideal and is an
    Einstein nilradical (the corresponding solvmanifold can be taken as the product of the real and the complex
    hyperbolic space).
    \item $p=D$, the free two-step nilpotent algebra; $\n$ is an Einstein nilradical by
    \cite[Proposition~2.9]{GK}.
    \item $p=D-1$; any such algebra is an Einstein nilradical by Proposition~\ref{p:p=D-1} in Section~\ref{s:dual}.
    \item $q \le 5, \; (p,q) \ne (5,5)$. The classification of these algebras is given in \cite{GT}. All of them,
    apart from few exceptions, are Einstein nilradicals \cite[Proposition~2]{Ni4}.
    \item $p=2, \; p=D-2$. The classification of Einstein nilradicals with $p=2$ is given in Theorem~\ref{t:twostep}.
    The case $p=D-2$ can be studied along the similar lines, but may be very computationally involved.
\end{itemize}

The main result of the paper is the classification of two-step Einstein nilradicals of type $(2,q)$ given in
Theorem~\ref{t:twostep} below. Our interest is motivated by two facts: firstly, the type  $(2,q)$ is one of the
exceptions in Proposition~\ref{p:twostepopen}, and secondly,
this class of nilpotent Lie algebras is one of the rare classes, for which an explicit
classification is at all possible (see Section~\ref{ss:2step} for details).

A two-step nilpotent Lie algebra $\n$ of type $(2,q)$ is defined by a point in the
space $\mathcal{V}(2,q)$ of pairs $(J_1, J_2)$ of skew-symmetric $q \times q$-matrices. The
$\GL(q)$-inva\-riants of such a pair are the elementary divisors and the minimal indices of the pencil
$xJ_1+yJ_2, \; x, y \in \br$.
We use the \emph{reduced} elementary divisors obtained by destroying all the
elementary divisors of the form $y^l$ and by taking every elementary divisor half of the times it appears (as the
pencil is skew-symmetric, its elementary divisors come in duplicate;  see Section~\ref{ss:notation} for details).
We say that a pencil is \emph{subsingular}, if all its elementary divisors are of the form
$(x+a_iy)^{l_i},\; i=1, \ldots, u$ (no complex elementary divisors), and the set $\{a_1, \ldots, a_u\}$ consists
of at most two elements. In the subsingular
case, relabel the elementary divisors in such a way
that $a_1=\ldots=a_{u'} \ne a_{u'+1}=\ldots=a_{u},$ $u=u'+u'', \; u,u',u''\ge 0$, and that
\begin{equation}\label{eq:lex}
(\max\nolimits_{1 \le i \le u'}l_i, u') \ge (\max\nolimits_{1 \le i \le u''}l_{u'+i}, u'')
\end{equation}
in the lexicographic order (where the maximum over the empty set is defined to be zero).

\begin{theorem}\label{t:twostep}
Let a two-step nilpotent Lie algebra $\n$ of type $(2,q)$ be defined by a pencil $xJ_1+yJ_2$ whose reduced
elementary divisors are $(x+a_iy)^{l_i}, \; 1 \le i \le u$, and
$((x+\mu_jy)^2+(\nu_jy)^2)^{n_j}, \; \nu_j \ne 0,\; 1 \le j \le w$, and whose minimal indices are $k_1, \ldots, k_v$,
where $u,w, v \ge 0$. Consider two cases:

\begin{enumerate} [\rm (A)]
    \item
    Generic case: either $w > 0$ or $\#\{a_1, \ldots, a_u\} \ge 3$.

    The algebra $\n$ is an Einstein nilradical if and only if
    \begin{enumerate}[\rm(i)]
        \item
        $l_i=n_j=1$, for all $i= 1, \ldots, u, \; j = 1, \ldots, w$, and
        \item
        for every $x \in \br, \quad \#\{i: a_i = x\} <  \frac12 u + w$.
    \end{enumerate}

    \item
    Subsingular case: $w = 0$ and $\#\{a_1, \ldots, a_u\} \le 2$. Assume that the $a_i$'s are labelled according
      to \eqref{eq:lex}.
    The algebra $\n$ is an Einstein nilradical if and only if $l_i=1$, for all $i=u'+1, \ldots, u$, and one of the
    following conditions holds:
    \begin{enumerate}[\rm(a)]
    \item
    $S_1=0$ and $l_i=1$, for all $i=1, \ldots, u'$;
    \item
    $S_1>0$ and $2k_j^2, \left[\tfrac12 (l_i^2+1)\right] < 2S_1^{-1}S_2$, for all $j=1, \ldots, v$ and all
    $i=1, \ldots, u'$,
    \end{enumerate}
    where
\begin{equation}\label{eq:asc}
    S_1= \sum\nolimits_{i=1}^{u'} l_i - u'', \quad
    S_2=1+\sum\nolimits_{i=1}^{u'} \tfrac16 (2l_i^3+l_i)+\sum\nolimits_{j=1}^v \tfrac16 k_j(k_j+1)(2k_j+1).
\end{equation}

\end{enumerate}

\end{theorem}

Informally, Theorem~\ref{t:twostep} says that in the generic case, an algebra $\n$ is an Einstein nilradical,
if it is defined by a matrix pencil having no nilpotent blocks in the canonical form and no elementary divisors of a
very high multiplicity.

\bigskip

The paper is organized as follows. Section~\ref{s:facts} gives the background on Einstein solvmanifolds and on
two-step nilpotent Lie algebras. Theorem~\ref{t:twostep} is proved in Section~\ref{s:pf}.
In Section~\ref{s:dual}, we discuss the connection between the property to be
an Einstein nilradical and the duality, and prove Proposition~\ref{p:p=D-1}.

\section{Preliminaries}
\label{s:facts}

\subsection{Einstein solvmanifolds}
\label{ss:es}

Let $G$ be a Lie group with a left-invariant metric $Q$
obtained by the left translations from an inner product $\ip$ on the Lie algebra $\g$ of $G$.
Let $B$ be the Killing form of $\g$, and let $H \in \g$ be the \emph{mean curvature vector} defined by
$\<H, X\> = \Tr \ad_X$.

The Ricci curvature $\mathrm{ric}$ of the metric Lie group $(G,Q)$ at the identity is given by
\begin{equation*}
    \mathrm{ric}(X)=-\<[H,X],X\>-\tfrac12 B(X,X)-\tfrac12 \sum\nolimits_i \|[X,E_i]\|^2
    +\tfrac14 \sum\nolimits_{i,j} \<[E_i,E_j],X\>^2,
\end{equation*}
for $X \in \g$, where $\{E_i\}$ is an orthonormal basis for $(\g, \ip)$
(\cite[Eq.~(2.3)]{Heb} or \cite[Eq.~(7.38)]{Bes}).

If $(\n, \ip)$ is a nilpotent metric Lie algebra, then $H = 0$ and $B = 0$, so
the Ricci operator $\ric$ of the metric Lie algebra $(\g, \ip)$ (the symmetric operator associated
to $\mathrm{ric}$) has the form
\begin{equation}\label{eq:riccinilexplicit}
\<\ric X, Y \> = \tfrac14 \sum\nolimits_{i,j} \<X, [E_i, E_j]\> \<Y, [E_i, E_j]\> -
\tfrac12 \sum\nolimits_{i,j} \<[X, E_i], E_j\> \<[Y, E_i], E_j]\>.
\end{equation}

By the result of \cite{La5}, any Einstein metric solvable Lie algebra is \emph{standard}, which means that
the orthogonal complement to the derived algebra $[\g, \g]$ is abelian.

It is proved in \cite{AK} that any Ricci-flat metric solvable Lie algebra is flat. By \cite{DM},
any Einstein metric solvable unimodular Lie algebra is also flat. We will therefore always assume $\g$
to be nonunimodular ($H \ne 0$), with an inner product of a strictly negative scalar curvature.

Any Einstein metric solvable Lie algebra admits a rank-one reduction \cite[Theorem 4.18]{Heb}. This means that if
$(\g, \ip)$ is such an algebra, with the nilradical $\n$ and the mean curvature vector $H$, then the
subalgebra $\g_1 = \mathbb{R}H \oplus \n$, with the induced inner product, is also Einstein. What is
more, the derivation $\Phi=\ad_{H|\n}:\n \to \n$ is symmetric with respect to the inner product, and all its
eigenvalues belong to $\gamma \mathbb{N}$ for some $\gamma > 0$. This implies, in particular, that the nilradical $\n$
of an Einstein metric solvable Lie algebra admits an $\mathbb{N}$-gradation defined by the eigenspaces of $\Phi$.
As proved in \cite[Theorem~3.7]{La1}, a necessary and sufficient condition for a metric nilpotent algebra
$(\n, \ip)$ to be the nilradical of an Einstein metric solvable Lie algebra is
\begin{equation}\label{eq:ricn}
    \ric = C \id_\n + \Phi, \quad \text{for some $\Phi \in \Der(\n)$},
\end{equation}
where $C \dim \g < 0$ is the scalar curvature of $(\g, \ip)$. This equation, in fact, defines $(\g, \ip)$, as explained
in Section~\ref{s:intro}. By \cite[Theorem 3.5]{La1}, a nilpotent Lie algebra admits no more than one
nil\-soli\-ton metric, up to conjugation  by $\Aut(\n)$ and scaling (and hence, an Einstein derivation, if it exists, is
unique, up to conjugation and scaling).
The set of eigenvalues $\la_i$ and their multiplicities $d_i$ of the Einstein derivation $\Phi$ of an Einstein nilradical
$\n$ is called the \emph{eigenvalue type} of $\n$ (and of $\Phi$). The eigenvalue type is usually written as
$(\la_1, \ldots, \la_p \, ; \, d_1, \ldots, d_p)$ (note that the $\la_i$'s are defined up to positive multiple).
Equation \eqref{eq:ricn}, together with \eqref{eq:riccinilexplicit}, implies that if $\n$ is an Einstein
nilradical, with $\Phi$ the Einstein derivation, then for some
$C <0, \quad \Tr (\Phi \, \psi) = - C \Tr \psi$, for any $\psi \in \Der(\n)$.
This motivates the following definition:

\begin{definition} \label{d:pE}
A derivation $\phi$ of a Lie algebra $\n$ is called \emph{pre-Einstein}, if it is semisimple, with all the eigenvalues
real, and
\begin{equation}\label{eq:pEtrace}
    \Tr (\phi \psi) = \Tr \psi,  \quad \text{for any $\psi \in \Der(\n)$}.
\end{equation}
\end{definition}

In practice, a pre-Einstein derivation can be found by solving a system of linear equations.
The usefulness of the pre-Einstein derivation for the study of Einstein nilradicals follows from the two
theorems below \cite[Theorems~1, 2]{Ni4}.

\begin{theorem} \label{t:preE}
Every nilpotent Lie algebra $\n$ admits a pre-Einstein derivation $\phi$, unique up to an automorphism of $\n$.
If $\n$ is an Einstein nilradical, its Einstein derivation is a positive multiple of $\phi$.
\end{theorem}

Choose and fix a pre-Einstein derivation $\phi$ of $\n$, and define the subalgebra
$\g_\phi = \{A \in \sl(\n) : [A, \phi] = 0,$ $\Tr(A \, \phi)=0\}$.
Let $G_\phi \subset \SL(n)$ be the connected Lie group with the Lie algebra $\g_\phi$.
The group $G_\phi$ is constructed as follows. Let $\n_j$ be the eigenspaces of $\phi$,
with the corresponding eigenvalues $\la_j, \; j=1, \ldots, p$. By \cite[Theorem~1]{Ni4}, all the 
$\la_j$'s are rational. Denote $\Lambda_j= N\la_j$, where
$N$ is the least common multiple of the denominators of the $\la_j$'s. Then $G_\phi$ is the
the subgroup of $\prod_{j=1}^p \GL(\n_j) \subset \GL(n)$ defined by
\begin{equation}\label{eq:Ggphi}
G_\phi=\{(g_1, \ldots, g_p) \, : \, g_j \in \GL^+(\n_j), \,
\prod\nolimits_{j=1}^p \det g_j = \prod\nolimits_{j=1}^p (\det g_j)^{\Lambda_j} = 1\},
\end{equation}
where $\GL^+(V)= \{g \in \GL(V) : \det g > 0 \}$.
As all the $\Lambda_j$'s are integers (although some could be zero or negative), $G_\phi$ is a real algebraic
reductive Lie group, with the Lie algebra $\g_\phi$.

Denote $\mathcal{V} = \wedge^2 (\Rn)^* \otimes \Rn$ the space of skew-symmetric bilinear
maps on~$\Rn$. Let $\mu$ be an element of $\mathcal{V}$ defining a nilpotent Lie algebra $\n= (\Rn, \mu)$.
Define the action of $G_\phi$ on the linear space $\mathcal{V}$
by $g.\nu(X,Y)=g\nu(g^{-1}X,g^{-1}Y)$ for $\nu \in \mathcal{V}, \, g \in G_\phi$. The following theorem from\
\cite{Ni4} gives a variational characterization of Einstein nilradicals:

\begin{theorem} \label{t:var}
A nilpotent Lie algebra $\n=(\Rn, \mu)$ with a pre-Einstein derivation $\phi$ is an Einstein nilradical if and only
if the orbit $G_\phi.\mu \subset \mathcal{V}$ is closed.
\end{theorem}

\subsection{Two-step nilpotent Lie algebras}
\label{ss:2step}
We follow \cite{Eb2, Gau}.
A two-step nilpotent Lie algebra $\n$ of dimension $p + q$, is said to be \emph{of type} $(p,q)$, if its derived
algebra $\m = [\n, \n]$ has dimension $p$. Clearly, $\m \subset \z(\n)$, the center of $\n$, and
$1 \le p \le D:=\frac12 q (q-1)$.

Choose a subspace $\mathfrak{b}$ complementary to $\m$ in $\n$ and two bases: $\{X_i\}$ for $\mathfrak{b}$
and $\{Z_\a\}$ for $\m$.
The Lie bracket on $\n$ defines (and is defined by) a $p$-tuple of skew-symmetric $q \times q$ matrices
$J_1, \ldots, J_p$ such that $[X_i, X_j] = \sum_{\a=1}^p (J_\a)_{ij} Z_\a$. The space of such $p$-tuples is
$\mathcal{V}(p,q)=(\wedge^2 \br^q)^p$. Note that the $J_\a$'s must be linearly independent, as $\m = [\n, \n]$,
so the points of $\mathcal{V}(p,q)$ corresponding to algebras of type $(p,q)$ form a subset
$\mathcal{V}^0(p,q) \subset \mathcal{V}(p,q)$, which is the complement to a real algebraic subset. The spaces
$\mathcal{V}(p,q)$ and $\mathcal{V}^0(p,q)$ are acted upon by the group $\GL(q) \times \GL(p)$ (change of bases):
for $x = (J_1, \ldots, J_p) \in \mathcal{V}(p,q)$ and $(M, T) \in \GL(q) \times \GL(p)$,
$(M, T).x= (\tilde J_1, \ldots, \tilde J_p)$, with $\tilde J_\a=\sum_{\beta=1}^p (T^{-1})_{\beta\a} M J_\beta M^t$.
Clearly, two points of $\mathcal{V}^0(p,q)$ lying on the same $\GL(q) \times \GL(p)$-orbit define isomorphic algebras.
The converse is also true, so that the space $\mathcal{X}(p, q)$ of the isomorphism classes of two-step nilpotent
Lie algebras of type $(p, q)$ is the quotient space $\mathcal{V}^0(p,q)/(\GL(q) \times \GL(p))$.
The space $\mathcal{X}(p, q)$ is compact, but in general is non-Hausdorff.

Relative to the basis $\{X_i, Z_\a\}$ for $\n$, any $\psi \in \Der(\n)$ is represented by a matrix of the form
$\left(\begin{smallmatrix} A_1 & 0 \\U & M \end{smallmatrix}\right)$, where $U \in \Mat(p,q)$ is arbitrary
(the set of the $\left(\begin{smallmatrix} 0 & 0 \\U & 0 \end{smallmatrix}\right)$'s is an abelian ideal in $\Der(\n)$),
and the matrices $A_1 \in \Mat(q,q)$ and $M \in \Mat(p,p)$ satisfy
  \begin{equation}\label{eq:der2step}
    J_\a A_1 + A_1^t J_\a = \sum\nolimits_{\beta=1}^p M_{\a\beta} J_\beta.
  \end{equation}

Let $\n$ be a two-step nilpotent Lie algebra of type $(p,q)$, with $1 \le p < D$, defined by a point
$x = (J_1, \ldots, J_p) \in \mathcal{V}^0(p,q)$. Choose an arbitrary basis $J'_\a, \; \a= 1, \ldots, D-p$, in the
orthogonal complement to the subspace $\Span(J_1, \ldots, J_p) \subset \wedge^2 \br^q$ with respect to the inner product
$Q(K_1, K_2)= -\Tr (K_1 K_2)$ on $\wedge^2 \br^q$. The point $x'=(J_1', \ldots, J_{D-p}') \in \mathcal{V}^0(D-p,q)$
defines a two-step nilpotent Lie algebra $\n^*$ of type $(D-p,q)$, which is called the \emph{dual} to $\n$. It is
easy to check that the isomorphism class of $\n^*$ is well-defined (depends only on the isomorphism class of $\n$).

The classification of two-step nilpotent Lie algebras of type $(p,q)$, in the meaning
``an explicit description of the orbit space $\mathcal{X}(p, q)$", is known only
for a very short list of
pairs $(p,q)$:
\begin{itemize}
  \item $p=D$, the free two-step nilpotent Lie algebra;
  \item $p=1$ (and the dual case $p=D-1$), the direct sum of the Heisenberg algebra and an abelian algebra (respectively
    the dual to such an algebra);
  \item $p=2$ (and the dual case $p=D-2$), see Section~\ref{ss:notation};
  \item $q \le 5$ or $(p,q)=(3,6), (12,6)$; the group $\SL(q)\times \SL(p)$ is ``visible" (see \cite{GT}).
\end{itemize}
This list might well exhaust all the cases when the classification is at all possible.
For instance, by the result of \cite{BLS}, the classification of two-step nilpotent algebras of type $(3,q)$, with an
arbitrary $q$, is hopeless.

\bigskip

In the proof of Theorem~\ref{t:twostep}, we will use two facts from the lemma below.

\begin{lemma} \label{l:twostepnil}
Let $\n= \b \oplus \m=(\br^{p+q}, \mu)$ be a two-step nilpotent Lie algebra of type $(p,q)$ defined by
a point $x = (J_1, \ldots, J_p) \in \mathcal{V}^0(p,q)$ relative to some bases $\{X_i\}$ for
$\mathfrak{b}$ and $\{Z_\a\}$ for $\m$.
\begin{enumerate}[\rm 1.]
    \item
    Suppose that for some $C \in \br$ and some diagonal derivation
$\phi =\left(\begin{smallmatrix} \phi_1 & 0 \\0 & \phi_2 \end{smallmatrix}\right) \in \Der(\n)$,
\begin{equation}\label{eq:twosteEinstein}
    \sum\nolimits_{\a=1}^p J_\a J_\a^t= - 2 C (I_q-\phi_1), \quad
    \Tr J_\a J_\beta^t = 4 C \K_{\a\beta} (1 - (\phi_2)_{\a\a}),
\end{equation}
Then $\n$ is an Einstein nilradical, with a pre-Einstein derivation $\phi$.
Moreover, the inner product $\ip$, for which the basis $\{X_i, Z_\a\}$ is orthonormal, is nilsoliton,
and the derivation $C\phi$ is Einstein.

    \item
Let $g = \left(\begin{smallmatrix} g_1 & 0\\ 0 & g_2 \end{smallmatrix}\right) \in \GL(\n)$, where
$g_1 \in \Mat(q,q)$ and $g_2 \in \Mat(p,p)$. Then $(\br^{p+q}, g.\mu)$ is a two-step
nilpotent Lie algebra of type $(p,q)$, and is defined by $\tilde x= (\tilde J_1, \ldots, \tilde J_p)$
such that
\begin{equation}\label{eq:Jchange}
\tilde J_\a=\sum\nolimits_{\beta=1}^p (g_2)_{\a\beta} g_1^{-1t} J_\beta g_1^{-1}.
\end{equation}

\end{enumerate}

\end{lemma}
\begin{proof}
1. As it follows from \eqref{eq:riccinilexplicit}, the Ricci tensor of $(\n, \ip)$ is given by
\begin{equation*}
    \ric_{|\b} = -\tfrac12 \sum\nolimits_{\a=1}^p J_\a J_\a^t, \quad
    (\ric_{|\m})_{\a\beta} = \tfrac14 \Tr J_\a J_\beta^t.
\end{equation*}
If the equations \eqref{eq:twosteEinstein} are satisfied for some derivation $\phi$, then, by \eqref{eq:ricn}
(with $\Phi=-C \phi$), the inner product is nilsoliton, so $\n$ is an Einstein nilradical.

2. Directly follows from the definition of $g.\mu$.
\end{proof}

\begin{remark}\label{rem:typical}
By Proposition~\ref{p:twostepopen}, a typical two-step nilpotent Lie algebra is an Einstein nilradicals of eigenvalue
type $(1,2; q,p)$.
Although the set of exceptional pairs $(p,q)$ in \cite{Eb3} is smaller, the proof of Proposition~\ref{p:twostepopen}
given \cite[Section~5]{Ni4} shows that a generic point of $\mathcal{V}(p,q)$ defines an Einstein nilradical with the
eigenvalue type $(1,2;q,p)$ in all the cases except for $(p,q)=(2, 2k+1)$ and its dual, which is precisely the statement
of \cite[Proposition~7.9]{Eb3}
(note that by \cite[Proposition~2.9(v)]{GK}, there are no two-step Einstein nilradicals of type $(1,2;2k+1,2)$ at all ).
In Proposition~\ref{p:twostepopen}, we narrow the dimension range
to exclude those cases when some algebras of type $(p,q)$ have open orbits in $\mathcal{V}(p,q)$ (and additionally,
the cases  $(p,q)=(2, 2k),\; k >3$). The remaining cases give a reasonable notion of being
typical not only in the linear space $\mathcal{V}(p,q)$, but also in the non-Hausdorff space $\mathcal{X}(p,q)$ of
isomorphism classes.
\end{remark}

\section{Proof of Theorem~\ref{t:twostep}}
\label{s:pf}

We start with a brief sketch of the proof.

In Section~\ref{ss:notation}, we define the elementary divisors and the minimal indices and give the
canonical form of the pencil $xJ_1+yJ_2$. We split the set of the pencils $xJ_1+yJ_2$ into the three cases:
Case 2 corresponds to the subsingular case in Theorem~\ref{t:twostep}, Case 1 and Case
3 combined give the generic case.

The proof in Case 1 (Section~\ref{ss:c1}) goes as follows. First, we compute the pre-Einstein
derivation $\phi$ (Lemma~\ref{l:precase1}). The group $G_\phi$ from Theorem~\ref{t:var}
appears to be too large, so we restrict ourselves to a smaller group, which acts diagonally
on the basis of $\b$ (the one, relative to which the pencil has the canonical form \eqref{eq:can3}), and which
acts as $\SL(2)$ on $\m$. Using Lemma~\ref{l:twostepnil} we show that $\n$ is an Einstein nilradical if
a particular function $F:\SL(2) \to \br$ has a
critical point on $\SL(2)$ (Lemma~\ref{l:Fh}). We further show that the function $F$ descends to a positive
geodesically convex function $\Phi$ on the hyperbolic space $\mathbb{H} = \SL(2)/\SO(2)$
(Lemma~\ref{l:Phiconvex}). If the conditions (i,ii) in (A) of Theorem~\ref{t:twostep}
are satisfied, the function $\Phi$ tends to infinity along any geodesic ray of $\mathbb{H}$, which
guarantees the existence of a critical point (in fact, a minimum). This proves the ``if"
direction of Theorem~\ref{t:twostep} for the pencils from Case 1. The ``only if" part is proved
by constructing a curve in the group $G_\phi$, along which the Lie bracket degenerates to
a non-isomorphic bracket (so that $\n$ is not an Einstein nilradical by Theorem~\ref{t:var}).

The proof of Theorem~\ref{t:twostep} for the algebras in Case 2 (Section~\ref{ss:c2}) is based on the fact
that every such algebra $\n$
has a \emph{nice basis}. According to Theorem~\ref{t:nice}, such an $\n$ is an Einstein nilradical
if and only if a certain convex geometry condition is satisfied. In the proof, we show
that this condition is equivalent to one of the conditions (a), (b) for the subsingular case
in Theorem~\ref{t:twostep}.

The proof for the algebras from Case 3 (Section~\ref{ss:c3}) can be reduced to that for the algebras
from Case~2 by \cite[Theorem~6]{Ni4}, which says that two real nilpotent Lie algebras whose
complexifications are isomorphic are or are not Einstein nilradicals simultaneously.

\subsection{Invariants and the canonical form}
\label{ss:notation}

The classification of two-step nilpotent Lie algebras of type $(2,q)$, both over $\br$ and over $\bc$, is well-known
and is based on the Kronecker theory of matrix pencils \cite{Gau,LR}. Namely, every algebra $\n$ of type $(2,q)$
is defined by a pair of linearly independent $q \times q$ skew-symmetric matrices $J_1, J_2$. To such a pair there
corresponds a skew-symmetric pencil $xJ_1+yJ_2$, with two pencils $xJ_1+yJ_2$ and $xJ'_1+yJ'_2$ defining the same algebra
if and only if they are \emph{projectively equivalent}, that is, if and only if there exist $P, Q \in \GL(q)$
such that the pencil $P(xJ_1+yJ_2)Q$ can be obtained from the pencil $xJ'_1+yJ'_2$ by a linear change of variables $x, y$
(if the pencils are equivalent, they are, in fact, \emph{congruent}: one can take $Q=P^t$; see \cite[Theorem~5.1]{LR}).

The complete set of \emph{affine} congruence invariants (under the action of $\GL(q)$ alone), is given by the
elementary divisors and the minimal indices. To define the elementary divisors of the pencil $xJ_1+yJ_2$, one computes
the invariant polynomials $i_r=\Delta_r/\Delta_{r-1}, \; r=1, \ldots, q$, where
$\Delta_0=1$, and $\Delta_r, \; r > 0$, is the greatest common divisor of the $r \times r$ minors of the matrix
$xJ_1+yJ_2$. Each of the invariant polynomials decomposes into the product of powers of prime polynomials. These
powers are called \emph{elementary divisors}, so that every elementary divisor over $\br$ has either the form
$(ax+by)^l, \; a^2+b^2 \ne 0$, or $(ax^2+2bxy+cy^2)^l,\; ac-b^2>0$. As we are interested in the projective equivalence,
we will slightly modify the elementary divisors as follows. Firstly, we destroy all the elementary divisors of the
form $y^l$ by replacing $J_1, J_2$ by their linear combinations, and then divide each of the elementary divisors by the
coefficient of the highest power of $x$. Then all the elementary divisors will be
either of the form $(x+a_iy)^{l_i}$, or of the form $((x+\mu_iy)^2+(\nu_iy)^2)^{n_i}, \nu_i \ne 0$. Secondly,
in the list of elementary divisors, some members may repeat. Moreover, as the pencil is skew-symmetric, every
elementary divisor repeats an even number of times \cite[Lemma~6.3]{LR}.
We take each of them half of the times it repeats. The resulting set
(more precisely, the multiset) of elementary divisors is called \emph{reduced}. The elementary divisors characterize
the ``regular part" of the pencil. For the ``singular part", one introduces the sequence of positive numbers $k_j$,
the \emph{minimal indices}. They are defined as the degrees of the elements of a homogeneous basis of the free submodule
$\Ker (xJ_1+yJ_2) \subset (\br[x,y])^q$ (see \cite[Definition~1.3]{GT}). 
The minimal indices and the reduced elementary
divisors (up to the linear change of variables $x, y$) form a complete set of projective invariants of a real
skew-symmetric pencil.

To give the canonical form of the pencil $xJ_1+yJ_2$ we introduce the following notation.
For $n \ge 1$, denote $I_n$ and $0_n$ the identity and the zero $n \times n$-matrices,
denote $N_n$ the nilpotent $n \times n$-matrix with $N_{i,i+1}=1$ for $1 \le i \le n-1$ and $N_{ij}=0$ otherwise, and
denote $L_n$ (respectively $R_n$) the $n \times (n+1)$-matrix obtained from the identity
matrix $I_n$ by attaching a column of $n$ zeros at the right (respectively at the left).

For rectangular matrices $A_i, \; i=1, \ldots, m$, we denote $A_1 \oplus \ldots \oplus A_m$ or $\oplus_{i=1}^m A_i$ the
block-diagonal matrix with the diagonal blocks $A_1, \ldots , A_m$ (in that order).
For an $m \times n$-matrix $A$, denote $\sk{A}$ the skew symmetric $(m+n) \times (m+n)$-matrix given by
$\sk{A}=\left(\begin{smallmatrix} 0_m & A \\ -A^t & 0_n \end{smallmatrix}\right)$.
For $n \ge 1$, denote $I^c_{2n}$ the block-diagonal $(2n \times 2n)$-matrix given by $I^c_{2n}=\oplus_{i=1}^n \sk{I_1}$.

We use the canonical form of a real skew-symmetric pencil $xJ_1 + y J_2$ given in \cite[Theorem~5.1]{LR} slightly
modified for our purposes. First of all, if the matrices $J_1, J_2$ have a common kernel (this corresponds to the zero
diagonal block $0_u$ in \cite[Eq.~(5.1)]{LR}), then the corresponding two-step nilpotent Lie algebra decomposes into the
direct sum of a two-step nilpotent Lie algebra of a smaller dimension and an abelian ideal. As an abelian Lie algebra
is an Einstein nilradical, \cite[Proposition~3.3]{La2} (or \cite[Theorem~7]{Ni4})
shows that adding it as a direct summand does not affect the property of a Lie algebra to be an Einstein nilradical.
Hence we lose no generality by assuming that $\Ker J_1 \cap \Ker J_2 = 0$, that is, that there are no zero blocks in
the canonical form.

Next, as our elementary divisors are reduced, the canonical form will have no blocks as on the second line
of \cite[Eq.~(5.1)]{LR} (the ``infinite" elementary divisors). On the other hand, if all the elementary divisors of
the pencil $xJ_1 + y J_2$ are of the form $(x+a_iy)^{l_i}$ and the number of the different $a_i$'s is at most
two, we can replace $J_1$ and $J_2$ by their appropriate linear combinations in such a
way that the resulting pencil has only the elementary divisors of the form $x^{l}$ and $y^{l}$.

Finally, permuting the rows and the columns in every block in \cite[Eq.~(5.1)]{LR}, we arrive at the following (disjoint)
cases for the canonical form:

\underline{Case 1.} The canonical form is
\begin{equation}\label{eq:can3}
\begin{gathered}
    J_1=\bigoplus\nolimits_{i=1}^u \sk{I_{l_i}} \oplus \bigoplus\nolimits_{i=1}^w \sk{I_{2n_i}}
    \oplus \bigoplus\nolimits_{j=1}^v \sk{L_{k_j}},
    \\
    J_2=\bigoplus\nolimits_{i=1}^u \sk{(a_i I_{l_i}+N_{l_i})}
    \oplus \bigoplus\nolimits_{i=1}^w \sk{(\mu_i I_{2n_i}+ \nu_i I^c_{2n_i} + N^2_{2n_i})}
    \oplus \bigoplus\nolimits_{j=1}^v \sk{R_{k_j}},
    \\
    \text{and $\#\{a_1, \ldots, a_u,
    \mu_1 + \mathrm{i} \nu_1, \mu_1 - \mathrm{i} \nu_1, \ldots, \mu_w + \mathrm{i} \nu_w, \mu_w - \mathrm{i} \nu_w\}
    \ge 3$}.
\end{gathered}
\end{equation}
In \eqref{eq:can3}, $\nu_i \ne 0, a_i, \mu_i$ are arbitrary real numbers (some of them could be equal to one
another), and $l_i, n_i, k_j \in \mathbb{N}$. Each of the numbers $u, v, w \ge 0$ is
allowed to be zero (in which case the corresponding blocks are omitted), as long as the condition on the third line
of \eqref{eq:can3} is satisfied.
The reduced elementary divisors of the pencil $xJ_1 + yJ_2$ over $\br$ are
$(x+a_iy)^{l_i},\; 1 \le i \le u$, and $((x+\mu_iy)^2+(\nu_iy)^2)^{n_i}, \; 1 \le i \le w$,
the minimal indices are $k_1, \ldots, k_v$.

\underline{Case 2.} The canonical form is
\begin{equation}\label{eq:can2}
    J_1=\bigoplus\nolimits_{i=1}^{u'} \sk{I_{l_i}} \oplus \bigoplus\nolimits_{i=u'+1}^u \sk{N_{l_i}}
    \oplus \bigoplus\nolimits_{j=1}^v \sk{L_{k_j}},
    \quad
    J_2=\bigoplus\nolimits_{i=1}^{u'} \sk{N_{l_i}} \oplus \bigoplus\nolimits_{i=u'+1}^u \sk{I_{l_i}}
    \oplus \bigoplus\nolimits_{j=1}^v \sk{R_{k_j}}.
\end{equation}
In \eqref{eq:can2}, $l_i, k_j \in \mathbb{N}$ and $u', u, v \ge 0$ (can be zero), with the only requirement that
$J_1$ and $J_2$ are linearly independent.
The (unreduced) elementary divisors of the pencil $xJ_1 + yJ_2$ are $x^{l_i}, \; 1 \le i \le u'$ and
$y^{l_i}, \; u'+1 \le i \le u$, and every elementary
divisor repeats twice, the minimal indices are $k_1, \ldots, k_v$.

\underline{Case 3.} The canonical form is
\begin{equation}\label{eq:canz}
    J_1=\bigoplus\nolimits_{i=1}^w \sk{I_{2n_i}} \oplus \bigoplus\nolimits_{j=1}^v \sk{L_{k_j}},
\qquad
    J_2=\bigoplus\nolimits_{i=1}^w \sk{(\mu I_{2n_i}+ \nu I^c_{2n_i} + N^2_{2n_i})}
    \oplus \bigoplus\nolimits_{j=1}^v \sk{R_{k_j}},
\end{equation}
where $w >0, \; v \ge 0, \; \nu \ne 0$, and $n_i, k_j \in \mathbb{N}$.
The reduced elementary divisors of the pencil $xJ_1 + yJ_2$ are $((x+\mu y)^2+(\nu y)^2)^{n_i}$, where $1 \le i \le w$,
the minimal indices are $k_1, \ldots, k_v$.

According to \cite[Theorem~5.1]{LR}, every pair of skew-symmetric matrices $J_1, J_2$, with
$\Ker \! J_1 \cap \Ker \! J_2 = 0$,
can be reduced to one of the forms \eqref{eq:can3}, \eqref{eq:can2},  or \eqref{eq:canz}. Moreover, the numbers
$u, v, w$ and $l_i, n_i, k_j$, up to permutation, ($v, w$ and $n_i, k_j$ in Case~3 respectively) are not only the
affine, but also the projective inva\-riants of the pencil $xJ_1 + yJ_2$ in each of the cases.

\subsection{Case 1}
\label{ss:c1}
This case is a part of the generic case in Theorem~\ref{t:twostep}.
We follow the strategy described in \cite[Section~3]{Ni4}. As the first step, we find a pre-Einstein derivation
for $\n$.

\begin{lemma}\label{l:precase1}
Let $\n$ be a two-step nilpotent Lie algebra of type $(2,q)$ defined by the matrices $J_1, J_2$ of the form
\eqref{eq:can3}. The linear subspace $\t \subset \End(\n)$ defined by
\begin{multline}\label{eq:torus}
    \t=\Bigl\{\begin{pmatrix} A_1 & 0 \\ 0 & 0_2 \end{pmatrix} +
    \begin{pmatrix} \eta I_q & 0 \\ 0 & 2\eta I_2 \end{pmatrix},\;
    A_1= \bigoplus\nolimits_{i=1}^u (\beta_i I_{l_i}\oplus (-\beta_i) I_{l_i}) \oplus
    \bigoplus\nolimits_{i=1}^w (\gamma_i I_{2n_i}\oplus (-\gamma_i) I_{2n_i})
    \\ \oplus \bigoplus\nolimits_{j=1}^v (\delta_j I_{k_j}\oplus (-\delta_j) I_{k_j+1}) \; | \;
    \eta, \beta_i, \gamma_i, \delta_j \in \br\Bigr\}.
\end{multline}
is a maximal $\br$-torus in $\Der(\n)$ (a maximal abelian subalgebra consisting of $\br$-diagonalizable derivations).
A derivation $\phi \in \t$ with $\beta_i=\gamma_i=0, \, \eta =1+\sigma, \, \delta_j=(2k_j+1)^{-1} \sigma, \;
\sigma=-4(q+8-\sum_{j=1}^v\frac{1}{2k_j+1})^{-1}$ is a pre-Einstein derivation for $\n$:
\begin{equation}\label{eq:precase1}
    \phi= (1+\sigma)I_{q_R}
\oplus \bigoplus\nolimits_{j=1}^v ((1+\sigma+\tfrac{1}{2k_j+1}) I_{k_j}\oplus (1+\sigma-\tfrac{1}{2k_j+1}) I_{k_j+1})
    \oplus 2 (1+\sigma)I_{2},
\end{equation}
where $q_R=2\sum_{i=1}^u l_i+ 4\sum_{i=1}^w n_i$ is the dimension of the ``regular part" of the pencil $xJ_1+yJ_2$.
\end{lemma}

\begin{proof}
From \eqref{eq:der2step} it is easy to see that $\t \subset \Der(\n)$, so $\t$ is an $\br$-torus of derivations. Let
$\psi \in \Der(\n)$ be an $\br$-diagonalizable
derivation commuting with all the elements of $\t$. Then from \eqref{eq:torus},
\begin{equation*}
\psi= \bigoplus\nolimits_{i=1}^u (P_{1i} \oplus Q_{1i}) \oplus
    \bigoplus\nolimits_{i=1}^w (P_{2i} \oplus Q_{2i})
    \oplus \bigoplus\nolimits_{j=1}^v (P_{3j} \oplus Q_{3j}) \oplus M,
\end{equation*}
where $P_{1i}, Q_{1i} \in \Mat(l_i,l_i), \; P_{2i}, Q_{2i} \in \Mat(2n_i,2n_i),
\; P_{3j}  \in \Mat(k_j,k_j), \; Q_{3j} \in \Mat(k_j+1,k_j+1)$, and $M \in \Mat(2,2)$.

By \eqref{eq:der2step}, the fact that $\psi \in \Der \n$ is equivalent to the following system of matrix equations:
\begin{equation}\label{eq:psider}
    \left\{ \begin{array}{l}
    Q_{ai} + P_{ai}^t=m_{11} I+m_{12} B_{ai}\\ B_{ai}Q_{ai} + P_{ai}^tB_{ai}=m_{21} I+m_{22} B_{ai},
    \end{array} \right.  \; a=1,2, \quad
    \left\{ \begin{array}{l}
    L_{k_j}Q_{3j} + P_{k_j}^tL_{k_j}=m_{11} L_{k_j}+m_{12} R_{k_j}\\
    R_{k_j}Q_{3j} + P_{k_j}^tR_{k_j}=m_{21} L_{k_j}+m_{22} R_{k_j},
    \end{array} \right.
\end{equation}
where $B_{1i}=a_i I_{l_i}+N_{l_i}, \; i=1, \ldots, u, \;
B_{2i}=\mu_i I_{2n_i}+ \nu_i I^c_{2n_i} + N^2_{2n_i}, \; i=1, \ldots, w$, the $m_{ij}$'s are the entries of $M$,
and the identity matrices on the right-hand side of the first system of \eqref{eq:psider} have the corresponding
sizes (that is, $l_i$ when $a=1$, and $2n_i$ when $a=2$).

From the first system of \eqref{eq:psider}, we obtain $P_{ai}^t=m_{11} I+m_{12} B_{ai} - Q_{ai}$ and
$B_{ai}Q_{ai}-Q_{ai}B_{ai}= f(B_{ai})$, where $f(t)=m_{21} +(m_{22}-m_{11}) t - m_{12} t^2$. For any
$N \ge 0, \; \Tr (B_{ai}^N[B_{ai},Q_{ai}])=0$, hence for any polynomial $f_0, \; \Tr (f_0(t)f(t))_{|t=B_{ai}}=0$,
which implies that $f(B_{ai})$ is a nilpotent matrix for all $a=1,2$, and all  the corresponding $i$'s.

Then the condition $\#\{a_i, \mu_i \pm \mathrm{i} \eta_i\} \ge 3$
of \eqref{eq:can3} and the fact that $f$ is at most quadratic imply that $f=0$, that is, $m_{12}=m_{21}=0$ and
$m_{22}=m_{11}$, so $M= 2 \eta \, I_2$ for some $\eta \in \br$. Moreover, as $f=0$,
the matrix $Q_{ai}$ commutes with $B_{ai}$, for every $a=1,2$, and for every $i$.
As the elementary divisors of $B_{ai}$ over $\bc$ are coprime (there is just one of them, $(a_i-\la)^{l_i}$,
when $a=1$, and two complex conjugate, $(\mu_i \pm \mathrm{i} \nu_i-\la)^{n_i}$, when $a=2$), each $Q_{ai}$ is a
polynomial of $B_{ai}$ (see e.g. \cite[Chapter~8, \S3]{Gan}). As $\psi$ is $\br$-diagonalizable,
this implies that all the $Q_{ai}$'s are scalar matrices:
$Q_{1i}=(\eta -\beta_i) I_{l_i}, \; i=1, \ldots, u$, $Q_{2i}=(\eta -\gamma_i) I_{2n_i}, \; i=1, \ldots, w$, for some
$\beta_i, \gamma_i \in \br$. Then $P_{1i}=(\eta +\beta_i) I_{l_i}, \; P_{2i}=(\eta +\gamma_i) I_{2n_i}$.

Now, for every $k=k_j$, the second system of \eqref{eq:psider} gives
$L_k Q + P^t L_k=2\eta  L_k$, $R_k Q + P^t R_k=2\eta  R_k$, for the corresponding matrices $P$ and $Q$. Replacing
$P$ and $Q$ by $P-\eta I_k$ and $Q-\eta I_{k+1}$ respectively we obtain $L_k Q + P^t L_k= R_k Q + P^t R_k=0$.
It follows that $(Q)_{ij}=(-P^t)_{ij}, \; (Q)_{i+1,j}=(-P^t)_{i,j+1}$ for $1 \le i,j \le k$, and
$(Q)_{i,k+1}=(Q)_{i+1,1}=0$ for $1 \le i \le k$, which easily implies that $Q=\delta I_{k+1}$ and $P=-\delta I_k$
for some $\delta \in \br$.

Therefore, the torus $\t$ contains all the $\br$-diagonalizable derivations $\psi$ of $\n$, which commute with $\t$,
hence $\t$ is a maximal $\br$-torus in $\Der(\n)$.

As it follows from the proof of assertion 1 (b) of \cite[Theorem~1]{Ni4}, every maximal $\br$-torus $\t$ contains a
unique pre-Einstein derivation, and what is more, for a derivation $\phi \in \t$ to be pre-Einstein, it is sufficient
that \eqref{eq:pEtrace} is satisfied
for all $\psi \in \t$. The fact that the derivation $\phi$ given by \eqref{eq:precase1} indeed has that
property is a matter of a direct calculation.
\end{proof}

Our next step is to show that condition (i) in (A) of Theorem~\ref{t:twostep} is necessary
for a nilpotent Lie algebra $\n=(\br^{2+q},\mu)$ from Case~1 to be an Einstein nilradical. We will use
Theorem~\ref{t:var}. Introduce the diagonal matrices $D_n=\diag(1,2, \ldots, n)$ and
$D^c_{2n}=\diag(1,1,2,2, \ldots, n-1,n-1,n,n)$ and define
\begin{equation*}
A=A_1 \oplus 0_2, \quad A_1=\bigoplus\nolimits_{i=1}^u ((-D_{l_i})\oplus D_{l_i}) \oplus
    \bigoplus\nolimits_{i=1}^w ((-D_{2n_i})\oplus D_{2n_i}) \oplus 0_{q-q_R}.
\end{equation*}
Then $A$ commutes with the pre-Einstein derivation $\phi$ and $\Tr A = \Tr (A \phi) = 0$, so $A \in \g_\phi$,
the Lie algebra of the group $G_\phi$ given by \eqref{eq:Ggphi}.
The limit $\mu'=\lim_{t \to \infty} \exp(tA).\mu$ exists, and $\n'=(\br^{2+q},\mu')$ is
a two-step nilpotent Lie algebra of type $(2,q)$ defined by the matrices
$J'_\a=\lim_{t \to \infty} \exp(-tA_1) J_\a \exp(-tA_1)$ (by \eqref{eq:Jchange}).
It follows from \eqref{eq:can3} that $J_1'=J_1$, and $J_2'$ is obtained from $J_2$ by removing all the nilpotent
parts ($N_{l_i}$ and $N^2_{2n_i}$) from the corresponding blocks. Then the number of the elementary
divisors of the pencil $xJ_1'+yJ'_2$ is different from that of the pencil $xJ_1+yJ_2$, unless there are no nilpotent
blocks in $J_2$ (that is, unless $l_i=n_i=1$). As that number is a projective invariant of the pencil (and hence is an
isomorphic invariant of the algebra), the orbit $G_\phi.\mu$ is non-closed, hence $\n$ is not an Einstein nilradical by
Theorem~\ref{t:var}, unless $l_i=1, \; i=1, \ldots, u$, and $n_i=1, \; i=1, \ldots, w$.

From now on, we assume that $l_i=n_i=1$ in \eqref{eq:can3}. We want to show that $\n=(\br^{2+q},\mu)$ is an Einstein
nilradical if and only if condition (ii) is satisfied.

Fix an inner product $\ip$ on $\Rn$ such that the basis $\{X_1, \ldots, X_q, Z_1, Z_2\}$ (for which the matrices
$J_1, J_2$ have the form \eqref{eq:can3}) is orthonormal.

Let $\rho_k$ be the representation of the group $\SL(2)$ on the space $\br^{(k-1)}[x,y]$ of homogeneous polynomials of
degree $k-1$ in two variables ($\rho_k(h)$ is the change of variables by $h^{-1}$). Let $P_k(h)$ be the matrix of
$\rho_k(h)$ relative to the basis $x^{k-1}, x^{k-2}y, \ldots, y^{k-1}$, so that for
$h=\left(\begin{smallmatrix} a & b\\ c & d \end{smallmatrix}\right) \in \SL(2)$,
$\sum_{i=1}^k(P_k(h^{-1}))_{ij}x^{k-j}y^{j-1} = (ax+by)^{k-i}(cx+dy)^{i-1}$. Then
$P_k(h^{-1})(aL_k+bR_k)P_{k+1}(h)=L_k$ and $P_k(h^{-1})(cL_k+dR_k)P_{k+1}(h)=R_k$ (which easily follows from
multiplying the $ij$-th entry of the matrices on the left-hand side by $x^{k+1-j}y^{j-1}$ and then summing up by
$j=1, \ldots, k+1$). So the matrices $P_k(h^{-1})$ and $P_{k+1}(h)$ ``undo" the transformation
$(L_k, R_k) \to (aL_k+bR_k, cL_k+dR_k)$.

Let $\mathcal{G} \subset \GL(2+q)$ be the set of matrices of the form $g=g_1 \oplus h$, where
\begin{equation}\label{eq:subset}
\begin{gathered}
    g_1= \bigoplus\nolimits_{i=1}^u (x_i I_2) \oplus \bigoplus\nolimits_{i=1}^w (y_i I_4)
    \oplus \bigoplus\nolimits_{j=1}^v (\Xi_j P_{k_j}^t(h)\oplus \Theta_j P_{k_j+1}(h^{-1})), \\
    h \in \SL(2), \quad
    \Xi_j=\diag(\xi_1^{(j)}, \ldots, \xi_{k_j}^{(j)}), \quad \Theta_j=\diag(\eta_1^{(j)}, \ldots, \eta_{k_j+1}^{(j)}),
    \quad x_i, y_i, \xi_s^{(j)}, \eta_s^{(j)} \ne 0.
\end{gathered}
\end{equation}

We have the following lemma.
{
\begin{lemma}\label{l:Fh}
The inner product $\ip$ on the Lie algebra $(\br^{2+q},g.\mu)$ is nilsoliton for some $g \in \mathcal{G}$, if and only if
the function $F: \SL(2) \to \br$ defined by
\begin{equation}\label{eq:Fh}
F(h)= \prod\nolimits_{i=1}^u \Tr (h^th \left(\begin{smallmatrix} 1 \\ a_i \end{smallmatrix} \right)
\left(\begin{smallmatrix} 1 \\ a_i \end{smallmatrix} \right)^t)
\prod\nolimits_{i=1}^w ( \Tr (h^th \left(\begin{smallmatrix} 1 & 0 \\ \mu_i & \nu_i \end{smallmatrix} \right)
\left(\begin{smallmatrix} 1 & 0 \\ \mu_i & \nu_i \end{smallmatrix} \right)^t))^2.
\end{equation}
has a critical point.
\end{lemma}

\begin{proof} The proof is essentially a direct computation.
By \eqref{eq:Jchange}, for $g \in \mathcal{G}$, with
$h=\left(\begin{smallmatrix} a & b\\ c & d \end{smallmatrix}\right) \in \SL(2)$, the algebra $(\br^{2+q},g.\mu)$
is defined by the matrices
\begin{equation}\label{eq:case1gmu}
\begin{gathered}
    J_1=\bigoplus\nolimits_{i=1}^u \sk{(x_i^{-2}(a+ba_i)I_1)} \oplus
    \bigoplus\nolimits_{i=1}^w \sk{(y_i^{-2}((a+ b \mu_i) I_2+ b \nu_i I^c_2))} \oplus
    \bigoplus\nolimits_{j=1}^v \sk{(\Xi_j^{-1}L_{k_j}\Theta_j^{-1})},
    \\
    J_2=\bigoplus\nolimits_{i=1}^u \sk{(x_i^{-2}(c+da_i)I_1)} \oplus
    \bigoplus\nolimits_{i=1}^w \sk{(y_i^{-2}((c+ d \mu_i) I_2+ d \nu_i I^c_2))} \oplus
    \bigoplus\nolimits_{j=1}^v \sk{(\Xi_j^{-1}R_{k_j}\Theta_j^{-1})}.
\end{gathered}
\end{equation}
By Lemma~\ref{l:twostepnil}, the inner product $\ip$ on the Lie algebra $(\br^{2+q},g.\mu)$ is nilsoliton if and only
if for the matrices $J_1, J_2$ given by \eqref{eq:case1gmu}, $J_1 J_1^t + J_2 J_2^t = - 2 C (I_q-\phi_1)$ and
$\Tr J_\a J_\beta^t = 4 C \K_{\a\beta} (1 - (\phi_2)_{\a\a})$, where $\phi=\phi_1 \oplus \phi_2$ is the pre-Einstein
derivation found in Lemma~\ref{l:precase1} and $C <0$.
Solving the equation $J_1 J_1^t + J_2 J_2^t = - 2 C (I_q-\phi_1)$ we find
\begin{equation}\label{eq:xiyi}
\begin{gathered}
x_i^4=(2C\sigma)^{-1} ((a+ba_i)^2+(c+da_i)^2), \quad
y_i^4=(2C\sigma)^{-1} ((a+b\mu_i)^2+b^2\nu_i^2+(c+d\mu_i)^2+d^2\nu_i^2),\\
(\xi_s^{(j)} \theta_s^{(j)})^2=(2k_j+1)(4C\sigma(k_j+1-s))^{-1}, \quad
(\xi_s^{(j)} \theta_{s+1}^{(j)})^2=(2k_j+1)(4C\sigma s)^{-1}, \quad s=1, \ldots, k_j.
\end{gathered}
\end{equation}
As $\sigma=-4(q+8-\sum_{j=1}^v(2k_j+1)^{-1})^{-1} <0$ and $ad-bc=1$, all the expressions on the right-hand side
are positive. Moreover, the equations on the second line of \eqref{eq:xiyi} give a linear system
for $\ln\xi_s^{(j)}, \ln \theta_s^{(j)}$, from which one can easily find $\xi_s^{(j)}$ and $\theta_s^{(j)}$.

Substituting \eqref{eq:xiyi} to the equations
$\Tr J_\a J_\beta^t = 4 C \K_{\a\beta} (1 - (\phi_2)_{\a\a}), \; \a,\beta=1,2$, we find that
the inner product $\ip$ on $(\br^{2+q},g.\mu), \; g=g_1 \oplus h \in \mathcal{G}$, is nilsoliton if and only if the
the entries of  $g_1$ are given by (\ref{eq:subset}, \ref{eq:xiyi}), and
$h=\left(\begin{smallmatrix} a & b\\ c & d \end{smallmatrix}\right) \in \SL(2)$ satisfies the following
system of equations:
\begin{equation*}
\begin{split}
    &\sum\nolimits_{i=1}^u 2\tfrac{(a+ba_i)^2}{(a+ba_i)^2+(c+da_i)^2}+
    \sum\nolimits_{i=1}^w 4\tfrac{(a+b\mu_i)^2+b^2\nu_i^2}{(a+b\mu_i)^2+b^2\nu_i^2+(c+d\mu_i)^2+d^2\nu_i^2} \\=&
    \sum\nolimits_{i=1}^u 2\tfrac{(c+da_i)^2}{(a+ba_i)^2+(c+da_i)^2}+
    \sum\nolimits_{i=1}^w 4\tfrac{(c+d\mu_i)^2+d^2\nu_i^2}{(a+b\mu_i)^2+b^2\nu_i^2+(c+d\mu_i)^2+d^2\nu_i^2} =
    C \sigma (2u + 4w), \\
    &\sum\nolimits_{i=1}^u 2\tfrac{(a+ba_i)(c+da_i)}{(a+ba_i)^2+(c+da_i)^2}+
    \sum\nolimits_{i=1}^w 4\tfrac{(a+b\mu_i)(c+d\mu_i)+bd\nu_i^2}{(a+b\mu_i)^2+b^2\nu_i^2+(c+d\mu_i)^2+d^2\nu_i^2} =0.
\end{split}
\end{equation*}
As $C<0$ can be chosen arbitrarily,
this system
is equivalent to the fact that the $2\times 2$ symmetric matrix
\begin{equation*}
d_h=
    \sum\nolimits_{i=1}^u 2
    \|h \left(\begin{smallmatrix} 1 \\ a_i \end{smallmatrix} \right)\|^{-2}
    (h \left(\begin{smallmatrix} 1 \\ a_i \end{smallmatrix} \right))
    (h \left(\begin{smallmatrix} 1 \\ a_i \end{smallmatrix} \right))^t+
    \sum\nolimits_{i=1}^w 4
    \Tr ((h \left(\begin{smallmatrix} 1 & 0 \\ \mu_i & \nu_i \end{smallmatrix} \right))
    (h \left(\begin{smallmatrix} 1 & 0 \\ \mu_i & \nu_i \end{smallmatrix} \right))^t)^{-1}
    (h \left(\begin{smallmatrix} 1 & 0 \\ \mu_i & \nu_i \end{smallmatrix} \right))
    (h \left(\begin{smallmatrix} 1 & 0 \\ \mu_i & \nu_i \end{smallmatrix} \right))^t
\end{equation*}
is proportional to the identity. The claim now follows from the fact that for every
$A \in \sl(2)$ and every $h \in \SL(2), \quad \frac{d}{dt}_{|t=0} (\ln F( \exp(tA) h)) = \Tr (d_hA)$.
\end{proof}
}

As it is immediate from \eqref{eq:Fh}, for every $U \in \SO(2),\quad F(Uh)=F(h)$, which implies that $F$ descends
to a function $\Phi$ on the homogeneous space $\mathbb{H}^2=\SL(2)/\SO(2)$ of the right cosets. The space
$\mathbb{H}^2$ equipped with
the right-invariant metric induced by the inner product $Q(A_1,  A_2) = \Tr (A_1 A_2^t)$ on $\sl(2)$ is isometric to
the hyperbolic plane $\mathbb{H}^2$. Let $\pi:\SL(2) \to \mathbb{H}^2$ be the natural projection (say, for the
Poincar\'{e} model in the half-plane $\mathrm{Im}(z) >0$, one can take
$\pi(\left(\begin{smallmatrix} a & b\\ c & d \end{smallmatrix}\right)) = \frac{d\mathrm{i}-b}{-c\mathrm{i}+a}$).
Then $\Phi: \mathbb{H}^2 \to \br$ is defined by $\Phi(z)=F(h)$, where $h \in \pi^{-1}(z)$.

{
\begin{lemma}\label{l:Phiconvex}
\begin{enumerate}[\rm 1.]
    \item
    The function $\Phi: \mathbb{H}^2 \to \br$ is positive and geodesically convex.
    \item
    If condition \emph{(ii)} of \emph{(A)} of Theorem~\ref{t:twostep} is satisfied, then $\Phi$ tends to infinity
    along any geodesic ray.
    If condition \emph{(ii)} is violated, then there exists a geodesic ray $\Gamma: \br^{+} \to \mathbb{H}^2$ such that
    $\lim_{t\to\infty}\Phi(\Gamma(t))$ exists and is finite.
\end{enumerate}
\end{lemma}
\begin{proof}
1. Consider the restriction of $\Phi$ to a geodesic of $\mathbb{H}^2$. A unit speed geodesic of $\mathbb{H}^2$ is given
by $\Gamma(t)=\pi(\exp \left(\frac12 t \left(\begin{smallmatrix} 1 & 0\\ 0 & -1 \end{smallmatrix}\right)\right) h)$,
where $h \in \SL(2)$. For
$h=\left(\begin{smallmatrix} a & b\\ c & d \end{smallmatrix}\right)$, we have
\begin{align*}
\Phi(\Gamma(t))&=F(\exp \left(\tfrac12 t \left(\begin{smallmatrix} 1 & 0\\ 0 & -1 \end{smallmatrix}\right)\right)h )
\\&=
\prod\nolimits_{i=1}^u (e^t (a+ba_i)^2 + e^{-t} (c+da_i)^2)
\prod\nolimits_{i=1}^w (e^t ((a+b\mu_i)^2 + b^2\nu_i^2)+ e^{-t} ((c+d\mu_i)^2 + d^2\nu_i^2))^2
\\&=
\sum\nolimits_{I=-u-2w}^{u+2w} A_I e^{It}.
\end{align*}
Since all the coefficients $A_I$ are nonnegative and at least one of them is positive, we obtain that $\Phi > 0$ and
$\frac{d^2}{dt^2}\Phi(\Gamma(t)) > 0$, so $\Phi$ is geodesically convex.

2. For a geodesic
$\Gamma(t)= \pi(\exp \left(\frac12 t \left(\begin{smallmatrix} 1 & 0\\ 0 & -1 \end{smallmatrix}\right)\right)
\left(\begin{smallmatrix} a & b\\ c & d \end{smallmatrix}\right))$, denote $n(\Gamma)=\#\{i: a+ba_i=0\}$. Without
loss of generality, we can assume that $a+ba_i=0$ for all $1 \le i \le n(\Gamma)$. Then $a+ba_i \ne 0$ for
$n(\Gamma) < i \le u$,
and $c+da_i \ne 0$ for $ 1 \le i \le n(\Gamma)$ (as $ad-bc=1$). Moreover, $(a+b\mu_i)^2 + b^2\nu_i^2 > 0$ for all
$i=1, \ldots, w$, as $\nu_i \ne 0$. Then $A_I = 0$ for all $I >2w+u-2n(\Gamma)$, and
$A_{2w+u-2n(\Gamma)} \ge \prod\nolimits_{i=1}^{n(\Gamma)} (c+da_i)^2
\prod\nolimits_{i=n(\Gamma)+1}^u (a+ba_i)^2$ $\times\prod\nolimits_{i=1}^w ((a+b\mu_i)^2 + b^2\nu_i^2)^2 > 0$.
It follows that $\lim_{t \to \infty}\Phi(\Gamma(t)) = \infty$ (respectively $\lim_{t \to \infty}\Phi(\Gamma(t))$
exists and is finite) if and only if $n(\Gamma) < \frac12 u + w$ (respectively $n(\Gamma) \ge \frac12 u + w$).

As the maximum of $n(\Gamma)$ taken over all the geodesic rays $\Gamma$ of $\mathbb{H}^2$ equals
$\max_{x \in \br}(\#\{i: a_i = x\})$, assertion 2 follows.
\end{proof}
}

Returning to the proof of the theorem, consider a sublevel set $S_M=\{z \in \mathbb{H}^2 : \Phi(z) \le M\}$,
where $M > \inf \Phi$.
The set $S_M$ is nonempty, closed, and geodesically convex by assertion~1
of Lemma~\ref{l:Phiconvex}. What is more, if condition (ii) of (A) of Theorem~\ref{t:twostep} is
satisfied, then $S_M$ is compact, as otherwise $S_M$ would contain a geodesic ray, 
which contradicts assertion~2 of Lemma~\ref{l:Phiconvex}. It follows that the function $\Phi: \mathbb{H}^2 \to \br$ has
a critical point, hence the function $F: \SL(2) \to \br$ has a critical point, which implies that the Lie
algebra $\n$ is an Einstein nilradical by Lemma~\ref{l:Fh}.

\bigskip

It remains to show that if condition (ii) is violated, then $\n=(\br^{2+q},\mu)$ is not an Einstein nilradical.
By Theorem~\ref{t:var}, it suffices to find a curve $g(t) \subset G_\phi$ such that the limit
$\mu^0=\lim_{t\to\infty}g(t).\mu$ exists, but the Lie algebra $\n^0=(\br^{2+q},\mu)$ is not isomorphic to $\n$.
A hint of how to construct such a $g(t)$ is given in assertion~2 of Lemma~\ref{l:Phiconvex}.

Let $m=\max_{x \in \br}(\#\{i: a_i = x\})$ and let $\xi \in \br$ be the value of $x$ at which the maximum is attained.
Without loss of generality, we can assume that $a_1=\ldots=a_m=\xi$ and $a_i \ne \xi$ for $i >m$.
Suppose that condition (ii) is not satisfied, that is, $m \ge \frac12 u + w$.

Define the curve $h(t) \subset \SL(2)$ by
\begin{equation}\label{eq:rayh}
h(t)=\exp \left(\frac12 t \left(\begin{matrix} 1 & 0\\ 0 & -1 \end{matrix}\right)\right)
\left(\begin{matrix} \xi & -1\\ 1 & 0 \end{matrix}\right)=
\left(\begin{matrix} \xi e^{t/2} & -e^{t/2}\\ e^{-t/2} & 0 \end{matrix}\right).
\end{equation}
Note that $\pi(h(t))$ is a geodesic ray in the hyperbolic plane $\mathbb{H}^2$, the restriction of the function $\Phi$
to which has a finite limit at infinity.

Introduce the curve $g(t)=g_1(t) \oplus h(t) \subset \mathcal{G} \subset \GL(2+q)$, where $h(t)$ is as
in \eqref{eq:rayh}, and $g_1(t)$ is given by \eqref{eq:subset}, with
$\Xi_j=I_{k_j}, \; \Theta_j=I_{k_j+1}, \; h=h(t),$ and the following $x_i$'s and $y_i$'s:
\begin{equation}\label{eq:xyray}
     y_i(t)=e^{t/4}, \; i=1, \ldots, w, \qquad
     x_i(t)=\left\{
                                                 \begin{array}{rl}
                                                   e^{-t/4}, & \hbox{$i=1, \ldots, 2w+u-m$;} \\
                                                   1, & \hbox{$i=2w+u-m+1, \ldots, m$;} \\
                                                   e^{t/4}, & \hbox{$i=m+1, \ldots, u$}
                                                 \end{array}
                                               \right.
\end{equation}
(note that the $x_i(t)$'s are well-defined, as $\frac12 u + w \le m \le u$), so that
\begin{equation*}
    g(t)=g_1(t) \oplus h(t) = \bigoplus\nolimits_{i=1}^u (x_i(t) I_2) \oplus  (e^{t/4} I_{4w})
    \oplus \bigoplus\nolimits_{j=1}^v (P_{k_j}^t(h(t))\oplus P_{k_j+1}(h(t)^{-1})) \oplus h(t).
\end{equation*}

The curve $g(t)$ lies in the group $G_\phi$ defined by \eqref{eq:Ggphi}. Indeed, as it follows from
\eqref{eq:precase1}, the pre-Einstein derivation $\phi$ has eigenvalues
$1+\sigma, 1+\sigma \pm \frac{1}{2k_j+1}$, and $2(1+\sigma)$, with the corresponding
eigenspaces $\n_{1+\sigma}$ of dimension $2u+4w$,
$\n_{1+\sigma + (2k+1)^{-1}}$ of dimension $2k \cdot \#\{j:k_j=k\}$,
$\n_{1+\sigma - (2k+1)^{-1}}$ of dimension $2(k+1) \cdot \#\{j:k_j=k\}$,
and $\n_{2(1+\sigma)}$ of dimension $2$ respectively. For all $t \in \br$, these eigenspaces are $g(t)$-invariant,
and moreover, $g(t) \in \SL(\n_{1+\sigma}) \times
\prod_{k}(\SL(\n_{1+\sigma + (2k+1)^{-1}}) \times \SL(\n_{1+\sigma - (2k+1)^{-1}})) \times \SL(2) \subset G_\phi$,
as $\prod_i x_i^2 \prod_i y_i^4 =1$ by \eqref{eq:xyray}, $h(t) \in \SL(2)$ by \eqref{eq:rayh}, and $P_k(h)$ is the
matrix of $\rho_k(h)$, where $\rho_k$ is a finite-dimensional representation of $\SL(2)$, so $\det P_k(h) =1$.

According to \eqref{eq:case1gmu} and \eqref{eq:Jchange}, the two-step nilpotent Lie algebra $(\br^{2+q}, g(t).\mu)$
is defined by the matrices
\begin{equation*}
\begin{gathered}
    J_1(t)=\bigoplus\nolimits_{i=1}^u \sk{(x_i(t)^{-2}e^{t/2}(\xi-a_i)I_1)} \oplus
    \bigoplus\nolimits_{i=1}^w \sk{((\xi- \mu_i) I_2 - \nu_i I^c_2)} \oplus
    \bigoplus\nolimits_{j=1}^v \sk{L_{k_j}},
    \\
    J_2(t)=\bigoplus\nolimits_{i=1}^u \sk{(x_i(t)^{-2} e^{-t/2}I_1)} \oplus
    \bigoplus\nolimits_{i=1}^w \sk{(e^{-t}I_2)} \oplus
    \bigoplus\nolimits_{j=1}^v \sk{R_{k_j}}.
\end{gathered}
\end{equation*}
As $a_1 = \ldots = a_m = \xi$ and $2w+u-m\le m$, the limits $\lim_{t\to\infty}J_\a(t):=J_\a^0$ exist, and
\begin{equation*}
\begin{gathered}
    J_1^0=0_{2m} \oplus \bigoplus\nolimits_{i=m+1}^u \sk{((\xi-a_i)I_1)} \oplus
    \bigoplus\nolimits_{i=1}^w \sk{((\xi- \mu_i) I_2 - \nu_i I^c_2)} \oplus
    \bigoplus\nolimits_{j=1}^v \sk{L_{k_j}},
    \\
    J_2^0=I^c_{2(2w+u-m)} \oplus 0_{2m} \oplus \bigoplus\nolimits_{j=1}^v \sk{R_{k_j}}.
\end{gathered}
\end{equation*}
The pencil $xJ_1^0 + y J_2^0$ has only two (unreduced) elementary divisors, $x$ and $y$ (both repeated $2u+4w-2m$ times),
and a $(4m-4w-2u)$-dimensional block of common zeros, so it is not
projectively equivalent to the pencil $xJ_1 + y J_2$ (in fact, it even belongs to Case 2 given by \eqref{eq:can2}).
Therefore, the algebra $\n^0=(\br^{2+q},\lim_{t\to\infty} g(t).\mu)$ is not isomorphic to $\n$, so $\n$ is
not an Einstein nilradical by Theorem~\ref{t:var}.
\qed

\begin{remark} \label{rem:eigen1}
In the cases when $\n$ is an Einstein nilradical, the corresponding eigenvalue type is given by the pre-Einstein
derivation found in Lemma~\ref{l:precase1} (note that, in contrast,  in the subsingular case the torus is ``bigger"
and a pre-Einstein derivation is substantially different). In particular, if the pencil $xJ_1+yJ_2$ has no
``singular part", the eigenvalue type of the corresponding Einstein nilradical is $(1,2;q,p)$.

As to the nilsoliton inner product, it can be constructed explicitly (by using (\ref{eq:subset}, \ref{eq:xiyi}) to
obtain $g \in \mathcal{G}$), provided one can explicitly find the critical points $h \in \SL(2)$ of the function $F$
from Lemma~\ref{l:Fh} (which seems doubtful).
\end{remark}

\subsection{Case 2}
\label{ss:c2}
This case corresponds to the subsingular case in Theorem~\ref{t:twostep}. The matrices $J_1$ and $J_2$ are given
by \eqref{eq:can2}, and by relabelling we can additionally assume that \eqref{eq:lex} holds: lexicographically,
$(\max\nolimits_{1 \le i \le u'}l_i, u') \ge (\max\nolimits_{1 \le i \le u''}l_{u'+i}, u'')$
(the maximum over the empty set is defined to be zero).
Some of the numbers $u, u', u''=u-u'$, and $v$ can be zero, but $\{J_1, J_2\}$ must be linearly independent.

We call a basis $\{e_1, \ldots, e_n\}$ for a nilpotent Lie algebra $\n$ \emph{nice}, if the structural constants
$c_{ij}^k$ relative to that basis satisfy the following two conditions:
for every $i,j, \quad \#\{k: c_{ij}^k \ne 0\}$ $\le 1$, and for every $i,k, \quad \#\{j: c_{ij}^k \ne 0\} \le 1$.
Given a nilpotent algebra of dimension $n$ with a nice basis, we fix an arbitrary ordering on the set
$\Lambda=\{(i,j,k) \, : \, c_{ij}^k \ne 0, \, i < j\}$ and introduce an $m \times n$ matrix $Y$ whose
$a$-th row has $1$ in columns $i$ and $j$, $-1$ in column $k$, and zero elsewhere, where
$1 \le a\le m=\#\Lambda$, and the $a$-th element of $\Lambda$ is $(i,j,k)$.
Denote $[\eta]_m$ an $m$-dimensional vector all of whose coordinates are equal to $\eta$.

\begin{theorem}[{\cite[Theorem~3]{Ni4}}]\label{t:nice}
A nonabelian nilpotent Lie algebra $\n$ with a nice basis is an Einstein nilradical if and only if
there exists a vector $\a \in \br^m$ with positive coordinates satisfying $YY^t \a = [1]_m$.
\end{theorem}

In our case, the algebra $\n$ has a nice basis, which is precisely the basis $\{X_1, \ldots, X_q,Z_1,Z_2\}$,
relative to which the matrices
$J_1$ and $J_2$ have the form \eqref{eq:can2}. To construct the matrix $Y$, we order the nonzero brackets
as follows: first the brackets of the form $[X_i, X_j]=Z_1, \; i<j$, in the increasing order of the $i$'s, and
then the brackets of the form $[X_i, X_j]=Z_2, \; i<j$, again in the increasing order of the $i$'s.

From \eqref{eq:can2} it follows that
\begin{equation}\label{eq:Y}
\begin{gathered}
    Y=\begin{pmatrix}
         M_1 & [1]_{N-u''} & [0]_{N-u''} \\
         M_2 & [0]_{N-u'} & [1]_{N-u'} \\
       \end{pmatrix}, \quad \text{where} \;
    N=\sum\nolimits_{i=1}^u l_i + \sum\nolimits_{j=1}^v k_j, \\
    M_1=\bigoplus\nolimits_{i=1}^{u'} (I_{l_i} \, | \, I_{l_i})
    \oplus \bigoplus\nolimits_{i=u'+1}^u (L_{l_i-1} \, | \, R_{l_i-1})
    \oplus \bigoplus\nolimits_{j=1}^v (I_{k_j} \, | \, L_{k_j}), \\
    M_2=\bigoplus\nolimits_{i=1}^{u'} (L_{l_i-1} \, | \, R_{l_i-1})
    \oplus \bigoplus\nolimits_{i=u'+1}^u (I_{l_i} \, | \, I_{l_i})
    \oplus \bigoplus\nolimits_{j=1}^v (I_{k_j} \, | \, R_{k_j}).
\end{gathered}
\end{equation}

It is easy to see that $\rk \left(\begin{smallmatrix} I_l & I_l \\ L_{l-1} & R_{l-1}\end{smallmatrix}\right) = 2l-1$,
and $\rk \left(\begin{smallmatrix} I_k & L_k \\ I_k & R_k \end{smallmatrix}\right) = 2k$, so
$\rk Y = m: = 2N-u$, the number of rows of $Y$, hence the equation
$YY^t \a = [1]_{m}$ has a unique solution $\a$. From \eqref{eq:Y}, we obtain
\begin{equation}\label{eq:YYtform}
\begin{gathered}
    YY^t=\begin{pmatrix}
         2 I_{N-u''}+ [1]_{N-u''}[1]_{N-u''}^t & M_3^t \\
         M_3 & 2 I_{N-u'}+ [1]_{N-u'}[1]_{N-u'}^t \\
       \end{pmatrix}, \\
    M_3=\bigoplus\nolimits_{i=1}^{u'} (L_{l_i-1} + R_{l_i-1})
    \oplus \bigoplus\nolimits_{i=u'+1}^u (L_{l_i-1}^t + R_{l_i-1}^t)
    \oplus \bigoplus\nolimits_{j=1}^v (I_{k_j} + R_{k_j}L_{k_j}^t).
\end{gathered}
\end{equation}
An explicit form of the vector $\a$ satisfying the equation $YY^t \a = [1]_{m}$ can be obtained as follows.
Let $\nu_1$ and $\nu_2$ be two real numbers, which we will define later.
Introduce the vectors $U^{1i}, V^{1i}, \; i=1, \ldots, u'$, $U^{2i}, V^{2i}, \; i=u'+1, \ldots, u$, and
$U^{3j}, V^{3j}, \; j=1, \ldots, v$, with the components
\begin{equation}\label{eq:pieces}
\begin{aligned}
  (U^{1i})_t &=(\nu_2-\nu_1)t^2+(\nu_1-\nu_2)(l_i+1)t+\tfrac12(\nu_2(l_i+1)-\nu_1l_i) & \quad t&=1, \ldots, l_i, \\
  (V^{1i})_t &=(\nu_1-\nu_2)(t^2-l_it)  & \quad t&=1, \ldots, l_i-1, \\
  (U^{2i})_t &=(\nu_2-\nu_1)(t^2-l_it)  & \quad t&=1, \ldots, l_i-1, \\
  (V^{2i})_t &=(\nu_1-\nu_2)t^2+(\nu_2-\nu_1)(l_i+1)t+\tfrac12(\nu_1(l_i+1)-\nu_2l_i) & \quad t&=1, \ldots, l_i, \\
  (U^{3j})_t &=(2 k_j + 1)^{-1} (k_j + 1 - t) ( t(\nu_1- \nu_2)(2 k_j  + 1) + \nu_2 (k_j + 1) - \nu_1k_j )
    & \quad t&=1, \ldots, k_j, \\
  (V^{3j})_t &=(2 k_j + 1)^{-1} t ((k_j + 1 - t) (\nu_2- \nu_1)(2 k_j  + 1) + \nu_1 (k_j + 1) - \nu_2k_j )
    & \quad t&=1, \ldots, k_j.
\end{aligned}
\end{equation}
Let $U$ be an $(N-u'')$-dimensional vector-column whose components are the components of $U^{11}$, followed by the
components of $U^{12}$, and so on, up to the components of $U^{3v}$, and let
$V$ be an $(N-u')$-dimensional vector-column whose components are the components of $V^{11}$, followed by the
components of $V^{12}$, and so on, up to the components of $V^{3v}$. Let $\a$ be an $m$-dimensional
vector-column whose components are the components of the vector $U$, followed by the components of the vector $V$.
A direct computation using (\ref{eq:YYtform}, \ref{eq:pieces}) shows that
$YY^t \a = ((\nu_1+\<U,[1]_{N-u''}\>)[1]_{N-u''}, (\nu_2+\<V,[1]_{N-u'}\>)[1]_{N-u'})^t$,
so it remains to choose $\nu_1$ and $\nu_2$ satisfying the equations
$\nu_1+\<U,[1]_{N-u''}\>=\nu_2+\<V,[1]_{N-u'}\>=1$.
From \eqref{eq:pieces},
\begin{equation}\label{eq:lambdas}
\begin{gathered}
\nu_1=\Delta^{-1}(1+2L_2+N_1+N_2+K), \quad \nu_2=\Delta^{-1}(1+L_1+L_2+2N_2+K), \quad \text{where}\\
L_1 =\sum\nolimits_{i=1}^{u'} \tfrac{l_i^3+2l_i}{6} , \; L_2 =\sum\nolimits_{i=1}^{u'} \tfrac{l_i^3-l_i}{6} , \;
N_1 =\sum\nolimits_{i=u'+1}^u \tfrac{l_i^3+2l_i}{6} , \; N_2 =\sum\nolimits_{i=u'+1}^u \tfrac{l_i^3-l_i}{6} , \\
K_1=\sum\nolimits_{j=1}^v \tfrac{k_j(k_j+1)(k_j^2+k_j+1)}{3 (2k_j+1)},  \;
K_2=\sum\nolimits_{j=1}^v \tfrac{k_j(k_j+1)(2k_j^2+2k_j-1)}{6 (2k_j+1)}, \;
K=K_1+K_2, \\ 
\Delta=(1+L_1+N_2+K_1)(1+L_2+N_1+K_1)-(L_2+N_2+K_2)^2.
\end{gathered}
\end{equation}
Since $L_1 > L_2 \ge 0$ (if $u'>0$), $N_1 > N_2 \ge 0$ (if $u''>0$), and $K_1 >K_2 > 0$ (if $v>0$), we obtain
$\Delta > 0$, hence $\nu_1, \nu_2 >0$.

According to Theorem~\ref{t:nice}, the algebra $\n$ is an Einstein nilradical if and only if all the components of the
vector $\a$ given by \eqref{eq:pieces} (with $\nu_1, \nu_2$ given by \eqref{eq:lambdas}) are positive. It is immediate
from \eqref{eq:pieces} that $(V^{1i})_t, \; t >0$, and $(U^{2i})_t, \; t >0$, cannot be simultaneously positive.
So for all the components of $\a$ to be positive, it is necessary that
$l_i=1$, for all $i=u'+1, \ldots, u$ (as the labelling satisfies \eqref{eq:lex}).
Then from \eqref{eq:lambdas}, $N_2=0, \; N_1 = \frac12 u''$.
Moreover, by \eqref{eq:pieces}, $(V^{2i})_1=\frac12 \nu_2> 0$,
for all $i=u'+1, \ldots, u$, and there are no components $U^{2i}$.

From \eqref{eq:pieces}, the components $(U^{3j})_t$ and $(V^{3j})_t$ are positive if and only if
\begin{equation}\label{eq:u3jv3jpos}
1-\tfrac12 k_j^{-2} < \nu_1^{-1} \nu_2, \quad 1-\tfrac12 k_j^{-2} < \nu_2^{-1} \nu_1,
\quad \text{for all $j=1, \ldots, v$}.
\end{equation}
For the components $(U^{1j})_t$ and $(V^{1j})_t$ we consider three cases.

If $u'=0$ (hence $u''=0$ by \eqref{eq:lex}), then from \eqref{eq:lambdas} $\nu_1=\nu_2$, so \eqref{eq:u3jv3jpos}
are satisfied, hence all the components of $\a$ are positive.

If $u' > 0$ and $l_i=1$, for all $i=1, \ldots, u'$ (hence $u' \ge u''$ by \eqref{eq:lex}), then
$(U^{1i})_1=\frac12 \nu_1> 0$ and there are no components $V^{1i}$. Therefore, $\n$ is an Einstein nilradical if and
only if \eqref{eq:u3jv3jpos} are satisfied.
From \eqref{eq:lambdas}, this is equivalent to $(u'-u'') \, k_j^2 < 1 + \frac12 u' +K$, that is, to
$S_1 k_j^2 < S_2$, where $S_1, S_2$ are given by \eqref{eq:asc}.
Depending on whether $S_1=u'-u''$ is zero or is positive, we obtain condition (a) or (b) from (B) of
Theorem~\ref{t:twostep} respectively.

If $u'>0$ and $l_i >1$ for some $i=1, \ldots, u'$, then $(V^{1i})_t > 0$ gives $\nu_2 > \nu_1$, that is,
$S_1=\sum_{i=1}^{u'} l_i - u''>0$ by \eqref{eq:lambdas}. From the inequalities  $(U^{1i})_t > 0$ we obtain
$\left[\frac12 (l_i^2+1)\right] < (\nu_2-\nu_1)^{-1} \nu_2$
(where the brackets denote the integer part), and from \eqref{eq:u3jv3jpos} we get
$2k_j^2 < (\nu_2-\nu_1)^{-1} \nu_2$. Substituting $\nu_1$ and $\nu_2$ from \eqref{eq:lambdas} we find that
$\nu_2-\nu_1= \tfrac12 \delta^{-1} S_1, \; \nu_2= \delta^{-1} S_2$. So $\n$ is an Einstein nilradical if and only if
condition (b) from (B) of Theorem~\ref{t:twostep} is satisfied.
\qed

\begin{remark} \label{rem:explicit}
As $\rk \, Y=m$, in all the cases when $\n$ is an Einstein nilradical, the corresponding Einstein metric solvable Lie
algebra can be constructed explicitly.
Namely, define a $(q+2)$-dimensional vector $s$ by $Ys=(\sqrt{\a_1}, \ldots, \sqrt{\a_{m}})^t$.
Then a nilsoliton inner product on $\n$ can be taken as
$\<X_a, X_b\>= \exp(2 s_a) \delta_{ab}, \; a, b =1, \ldots, q+2$, where we denote $X_{q+1}=Z_1$ and $X_{q+2}=Z_2$
(see \cite[Theorem~1]{Pay} or the paragraph after the proof of Theorem~3 in \cite[Section~4]{Ni4}).
The eigenvalue type can be also easily found from \cite[Eq.~(11)]{Ni4}.

Also note that a two-step nilpotent algebra $\n$ defined by a singular pencil $xJ_1+yJ_2$ is always an Einstein
nilradical (as it follows from the above proof when $u=u'=u''=0$).
\end{remark}

\subsection{Case 3}
\label{ss:c3}

According to  Theorem~\ref{t:twostep}, a nilpotent algebra $\n$ from this case is an Einstein nilradical
if and only if $n_i=1$, for all $i= 1, \ldots, w$.

To prove that, we use \cite[Theorem~6]{Ni4}, which says that two real nilpotent Lie algebras whose complexifications are
isomorphic are or are not Einstein nilradicals simultaneously.

The complexification of $\n$ is a two-step nilpotent Lie algebra $\n^{\bc}$ defined by two complex skew-symmetric
matrices $J_1, J_2$ such that complex pencil $x J_1 + y J_2$ has the minimal indices $k_1, \ldots, k_v$, and the
reduced elementary divisors $(x+ z y)^{n_i}$ and $(x+ \overline{z} y)^{n_i}, \; 1 \le i \le w$, where
$z = \mu + \mathrm{i} \nu \in \bc \setminus \br$.
Replacing $J_1$ and $J_2$ by $\frac{\mathrm{i}}{2\nu} (\overline{z} J_1 - J_2)$ and
$\frac{\mathrm{i}}{2\nu} (-z J_1 - J_2)$ respectively we obtain the complex matrix pencil $x J_1 + y J_2$ with the
same minimal indices and with the reduced elementary divisors $x^{n_i}, y^{n_i}, \; 1 \le i \le w$.

The canonical form of such a pencil over $\bc$ is given by \cite[Theorem~6.8]{Gau} and coincides (up to permuting rows
and columns) with the canonical form \eqref{eq:can2} of Case~2, with $u'=u'' >0$ and $l_i=l_{i+u'}$, $i=1, \ldots, u'$.
The corresponding two-step nilpotent Lie algebra $\tilde \n$ is defined over $\br$, with $\n^\bc \simeq \tilde \n^\bc$.

The claim of Theorem~\ref{t:twostep} for the algebras from Case 3 now follows from \cite[Theorem~6]{Ni4} and
the proof for Case~2 given above (specifically, from condition (a) of part (B) of Theorem~\ref{t:twostep}).
\qed

\section{Two-step Einstein nilradicals and the duality}
\label{s:dual}

Recall that the dual to a two-step nilpotent Lie algebra $\n$ of type $(p,q)$, with $1 \le p < D$, defined by a point
$(J_1, \ldots, J_p) \in \mathcal{V}^0(p,q)$ is the two-step nilpotent Lie algebra $\n^*$ of type $(D-p,q)$ defined by
the point
$(J'_1, \ldots, J'_{D-p}) \in \mathcal{V}^0(D-p,q)$, where $J_\a'$ are linearly independent $q \times q$ skew-symmetric
matrices such that $\Tr (J'_\a J_\beta)=0$, for all $\a=1, \ldots, p,\; \beta=1, \ldots, D-p$.

By \cite[Proposition~7.6]{Eb3} (see also \cite[Proposition~2.9(iv)]{GK}), the two-step nilpotent Lie algebra dual to
an Einstein nilradical with the eigenvalue type $(1,2; q,p)$ is an Einstein nilradical, with the eigenvalue type
$(1,2; q, D-p)$. In Proposition~\ref{p:p=D-1} below, we show that every two-step nilpotent Lie algebra of type $(D-1,q)$
is an Einstein nilradical, which can be compared with the fact that every two-step nilpotent Lie algebra of type $(1,q)$
is an Einstein nilradical. However in general, for an arbitrary eigenvalue type, the algebra dual to a two-step Einstein
nilradical is not necessarily an Einstein nilradical. The following proposition shows that that might happen
even for the algebras of a very simple structure (other examples can be found in \cite[Section~5]{Ni4}).

\begin{proposition}\label{p:freeabdual} 
Let a two-step nilpotent Lie algebra $\n$ be a direct sum of the free two-step nilpotent Lie algebra
$\f(f,2)$ on $f \ge 2$ generators, and an abelian ideal $\br^a, \; a \ge 1$. Then:
\begin{enumerate}[\rm (i)]
  \item $\n$ is always an Einstein nilradical;
  \item the dual algebra $\n^*$ is an Einstein nilradical if and only if $a \ge f$.
\end{enumerate}
\end{proposition}
\begin{proof}
Assertion (i) follows immediately by combining \cite[Proposition~2.9(iii)]{GK}
and \cite[Proposition~3.3]{La2}.

To prove (ii), one can use the fact that $\n^*$ has a nice basis (constructed in the obvious way), and then apply
Theorem~\ref{t:nice}. It is easy to see that $\rk Y = m$, where $m= fa+\frac12 a(a-1)$ is the number of rows of $Y$,
so the equation $YY^t \a = [1]_m$ has a unique solution. From the symmetries, it follows that $fa$ components of
$\a$ are equal to $a\delta$, and the remaining $\frac12 a(a-1)$ components are equal to $(a+1-f)\delta$, where
$\delta=(2a^2+a+f-1)^{-1} > 0$, which proves (ii).
\end{proof}

\begin{proposition}\label{p:p=D-1}
Any two-step nilpotent Lie algebra of type $(D-1,q)$ is an Einstein nilradical.
\end{proposition}

\begin{proof}
We will prove the proposition by explicitly constructing the nilsoliton inner product
(the proof is similar to that of Lemma~6 in the unpublished preprint \cite{Ni1}). Let $\n=\b\oplus\m$
be a two-step nilpotent Lie algebra of type $(D-1,q)$. Its dual algebra $\n^*$ is of type $(1,q)$ and hence is
defined by a single nonzero skew-symmetric $q \times q$ matrix $J= \sk{I_d} \oplus 0_l$, with $2d+l=q$ and $d >0$.

If $l=0$, then $\n^*$ is the Heisenberg algebra and is an Einstein nilradical with the eigenvalue type $(1,2; q, 1)$,
so the claim follows from \cite[Proposition~7.6]{Eb3}, \cite[Proposition~2.9(iv)]{GK}. We will therefore assume that
$l >0$. Let $Q_0$ be the inner product on $\wedge^2 \br^q$ defined by $Q_0(K_1, K_2) = -\Tr (K_1 K_2)$.

Consider the following three subspaces of $\wedge^2 \br^q$:
\begin{equation*}
\begin{gathered}
    L_1 = \{K \oplus 0_l \; | \; K \in \wedge^2 \br^{2d}, \Tr(\sk{I_d} K)=0  \}, \\
    L_2 = \{\sk{T} \; | \; T \in \Mat(2d,l)\}, \quad
    L_3 = \{0_{2d} \oplus N \; | \; N \in \wedge^2 \br^l  \}.
\end{gathered}
\end{equation*}
Then $(\br J)^{\perp}_0=L:=L_1 \oplus L_2 \oplus L_3$, and the algebra $\n$ is defined by an arbitrary basis of $L$.
Introduce an inner product $Q$ on $L$ by scaling $Q_0$ on each of the $L_i$'s:
$Q(K_1, K_2)= \K_{ij} r_i^{-2} Q_0(K_1, K_2)$, for $K_1 \in L_i$, $K_2 \in L_j, \; i,j=1,2,3, \; r_i > 0$, and
choose a basis $\{J_\a\}$ for $L$, which is orthonormal with respect to $\ip$ and is the union of bases
for $L_1, L_2$, and $L_3$.

Let $\{X_i\}$ and $\{Z_\a\}$ be the bases for $\b$ and $\m$ respectively such that
$[X_i,X_j]=\sum_{\a=1}^{D-1} (J_\a)_{ij}Z_\a$, and let $\ip$ be an inner product on $\n$ such that the basis
$\{X_i, Z_\a\}$ is orthonormal.

Then for $J_\a \in L_i,\, J_\beta \in L_j$, we have $\Tr J_\a J_\beta^t = r_i^2\K_{\a\beta}$ and, by an
easy computation,
$\sum\nolimits_{\a=1}^{D-1} J_\a J_\a^t= (\frac{2d^2-d-1}{2d}r_1^2+\frac{l}{2} r_2^2)I_{2d} \oplus
(\frac{l-1}{2}r_3^2+d r_2^2)I_l$.
Substituting this to \eqref{eq:twosteEinstein} we find that, relative to the basis
$\{X_i, Z_\a\}, \quad \phi=\la_1 I_{2d} \oplus \la_2 I_l \oplus \la_3 I_{d(2d-1)} \oplus \la_4 I_{2dl}
\oplus \la_5 I_{\frac12 l(l-1)}$, where $\la_1=1+(2c)^{-1}(\frac{2d^2-d-1}{2d}r_1^2+\frac{l}{2} r_2^2)$,
$\la_2=1+(2c)^{-1}(\frac{l-1}{2}r_3^2+d r_2^2)$, and $\la_i=1-(4c)^{-1}r_{i-2}^2$ for $i=3,4,5$. According
to Lemma~\ref{l:twostepnil}, the inner product $\ip$ on $\n$ is nilsoliton (and hence $\n$ is an Einstein nilradical),
if $\phi$ is a derivation.
By \eqref{eq:der2step}, this condition is equivalent to the system of equations
$2\la_1=\la_3, \; \la_1+\la_2=\la_4$, and $2\la_2=\la_5$ (with the latter equation omitted, if $l=1$).
Solving it, with the scaling constant $c=-\frac14 ((5+2q^2-3q)d+2-4q) <0$, we find
\begin{equation*}
r_1^2=d(q-1), \; r_2^2= qd-d-1, \; r_3^2=q d - d - 2.
\end{equation*}
The right-hand sides are always positive, with the only exception: $r_3=0$ when $d=l=1$, which does not cause any
problem, as $\dim L_3 = 0$ when $l=1$ (so we don't need $r_3$ anyway).
\end{proof}

\begin{remark} \label{rem:duo}
The following observation could be a possible first step in finding Einstein nilradicals among the two-step nilpotent
Lie algebras whose duals are well understood.
By \eqref{eq:der2step}, to every derivation $\psi=A_1 \oplus M \in \Der(\n)$ there corresponds a derivation
$\psi'=A_1^t \oplus M' \in \Der(\n^*)$. This correspondence is actually an isomorphism between the subalgebras
of $\Der(\n)$ and $\Der(\n^*)$. In particular, it maps the maximal tori of $\Der(\n)$ onto the maximal tori
of $\Der(\n^*)$. This may help to find a pre-Einstein derivation for $\n^*$ when a maximal torus of $\Der(\n)$
is known (note however, that the image of a pre-Einstein derivation is not necessarily a pre-Einstein derivation).
\end{remark}


\end{document}